\documentclass[11pt]{amsart}
\usepackage[foot]{amsaddr}
\usepackage{amssymb,amsmath}
\usepackage[alphabetic,initials,nobysame]{amsrefs}
\usepackage{eucal}
\usepackage{color}

\theoremstyle{plain}
\newtheorem{theorem}{Theorem}
\newtheorem{lemma}{Lemma}
\newtheorem{proposition}{Proposition}
\newtheorem{definition}{Definition}

\newtheorem{remark}{Remark}

\numberwithin{equation}{section}
\numberwithin{figure}{section}

\newcommand{\D}{\partial}
\newcommand{\R}{\mathbb{R}}
\newcommand{\N}{\mathbb{N}}

\title[Higher regularity of the free boundary in the parabolic Signorini problem]{Higher regularity of the free boundary in the parabolic Signorini problem}

\author[Agnid Banerjee]{Agnid Banerjee}
\address[AB]{University of California, Irvine, Department of Mathematics}
\email[Agnid Banerjee]{agnidb@uci.edu}
\thanks{}

\author[Mariana Smit Vega Garcia]{Mariana Smit Vega Garcia}
\address[MSVG]{University of Washington, Department of Mathematics, Padelford Hall, Seattle, WA 98195, USA, phone number +1 206-543-7146} 
\email[Mariana Smit Vega Garcia]{marianag@uw.edu}

\thanks{Second author supported in part by DFG, Projekt ``Singularit{\"a}ten ElektroHydroDynamischer Gleichungen".} 

\author[Andrew K. Zeller]{Andrew K. Zeller}
\address[AKZ]{Purdue University, Department of Mathematics}
\email[Andrew K. Zeller]{zellera@purdue.edu}
\thanks{}


\begin{document}
\begin{abstract}
We show that the quotient of two caloric functions which vanish on a portion of an $H^{k+ \alpha}$ regular slit is $H^{k+ \alpha}$ at  the slit,  for $k \geq 2$. In the case $k=1$, we show that the quotient is  in $H^{1+\alpha}$ if the slit is assumed to be space-time $C^{1, \alpha}$  regular. This can be thought of as a parabolic analogue of a recent important result in \cite{DSS14a}, whose ideas inspired us. As an application, we show that the free boundary near a regular point of the parabolic  thin obstacle problem studied in \cite{DGPT} with zero obstacle is $C^{\infty}$ regular in space and time.
\end{abstract}
\maketitle

\tableofcontents

\section{Introduction}\label{S:intro}
The classical comparison theorem states that two nonnegative harmonic functions which vanish on the boundary of a Lipschitz domain, or more generally an NTA domain, must vanish at the same rate. An important consequence of this result is that the quotient of two such functions is in fact H\"older continuous up to the boundary (with the added restriction that the function in the denominator needs to now be nonnegative). In their recent remarkable work, De Silva and Savin have established a higher-order version of this result. More specifically, they have proven in \cite{DSS} the following:

\begin{theorem}\label{t}
Let $D$ be a $C^{k, \alpha}$ domain in $\R^n$, with $0 \in \partial D$. Let $u, v$ be   two harmonic functions vanishing on $\partial D \cap B(0,1)$.  Furthermore, let  $u>0$ in $D$ and $u=1$ at some interior point in $D$. Then,
\begin{equation}\label{e:1}
\left\|\frac{v}{u}\right\|_{C^{k, \alpha}(B(0, 1/2))}  \leq C ||v||_{L^{\infty}(B(0,1))}.
\end{equation}
\end{theorem}

This result establishes regularity of the quotient one order higher than one might expect. Indeed, the classical Schauder estimates imply that $u, v$ are $C^{k, \alpha}$ up to the boundary. Then, by the Hopf Lemma, we have $u_{\nu} > 0$, from which one can assert that the quotient $\frac{v}{u}$ is $C^{k-1,\alpha}$ up to the boundary. However, Theorem \ref{t} remarkably states that the ratio is in fact $C^{k, \alpha}$ up to the boundary. The special case $k=0$ of this result is the \emph{boundary Harnack principle} mentioned above, see  \cite{CFMS} and \cite{JK}. Very recently, such a result has been generalized to the parabolic case in \cite{BG}.

\medskip

Besides being an interesting regularity result in its own right, a direct application of Theorem \ref{t} above implies $C^\infty$ smoothness of a priori $C^{1,\alpha}$ free boundaries for the classical obstacle problem with zero obstacle without the use of the hodograph transformation as in \cite{KN}, \cite{KNS}, a tool which has thus far been the standard way of establishing smoothness of free boundaries starting from $C^{1, \alpha}$. Having said this, we would like to mention that the hodograph transformation in \cite{KN}, \cite{KNS} does in fact imply real-analyticity of the free boundary, which is instead not implied by Theorem \ref{t}. Nevertheless, Theorem \ref{t} provides a new perspective in the study of Schauder theory and free boundary problems.

\medskip

Theorem \ref{t} was  subsequently generalized by De Silva and Savin to slit domains in \cite{DSS14a}. In order to state their result, we first introduce the following notations, which should not be confused with the ones we will use beginning in Section \ref{S:notations}.  We write the points of $\R^{n+1}$ as $X=(x,x_{n+1})$, where points of $\R^n$ are denoted with $x=(x',x_n)$, for $x' \in \R^{n-1}$.
We denote the $n-$dimensional slit in $\R^{n+1}$ by
 \[
\mathcal{P} = \{X \in \R^n\times \R \ | \  x_{n+1} =0, \ x_{n} \leq g(x') \},
\]
where $g$ is assumed to be in $C^{k+1+\alpha}$  for $k \geq 0$. In particular, we will assume  $g(0)=0$, $\nabla_{x'}g(0)=0$, and $\|g\|_{C^{k+1+\alpha}} \leq 1$. We also define
\[
\Gamma = \{X \in \R^n\times\R \ | \ x_{n+1} =0, \ x_{n}=g(x')\}.
\]
Given $X=(x, x_{n+1})$, let $d$ denote the signed distance in $\R^{n}$ from $x$ to $\Gamma$. Furthermore, let $r= \sqrt{x_{n+1}^2 + d^2}$.

\medskip

\begin{theorem}\label{T}
Let $k \geq 0$. Let $U>0$ be a solution of $\Delta U=0$ and $u$ be a solution of $\Delta u = \frac{U_0}{r}f$ in $B_1\setminus \mathcal{P}$ where $U_0 = \frac{1}{\sqrt{2}}\sqrt{d+r}$, $f \in C^{k+\alpha}_{x'r}(\Gamma \cap B_1)$ and $\|f\|_{C^{k+\alpha}_{x'r}(\Gamma \cap B_1)} \leq 1$. Assume that $U, u \in C(B_1)$, that both are even in $x_{n+1}$ and vanish continuously on $\mathcal{P}$. Furthermore, assume that $\Gamma \in C^{k+1+\alpha}$, $\|u\|_{L^\infty(B_1)} \leq 1$ and $U(\frac{1}{2}e_n)=1$. Then 
\[
\Big\|\frac{u}{U}\Big\|_{C^{k+1+\alpha}_{x'r}(\Gamma \cap B_{1/2})} \leq C,
\]
where $C=C(n,k,\alpha)$.
\end{theorem}

Similar to the case of the classical obstacle problem, a direct application of Theorem \ref{T} implies (see Theorem 1.2 in \cite{DSS14a}) smoothness of the free boundary near regular points for the thin obstacle, or Signorini, problem with zero obstacle studied in \cite{ACS} (see also \cite{AC}, \cite{CSS}). Note that real analyticity of the free boundary near regular  points  in the thin obstacle problem was recently  established in \cite{KPS} by using a method  based on the hodograph transformation.

\medskip

These recent results and their  applications to free boundary problems motivated us to investigate their parabolic counterpart. While difficulties arise in using the methods of \cite{KPS} in the parabolic setting, the methods of De Silva and Savin carry over much more naturally from the elliptic case. Our main result, Theorem \ref{T:main}, constitutes the parabolic analogue of Theorem \ref{T} above.  In the present  paper we make the observation that to generalize the ideas  in \cite{DSS14a} to the parabolic situation, one needs to make delicate adaptations due to different scalings of the space and time variables.  Similar to the elliptic case, an application of Theorem \ref{T:main} implies  smoothness of the free boundary  near  regular points for the parabolic Signorini problem with zero obstacle studied in \cite{DGPT}. However, we would like to mention that unlike in the elliptic case, in order to apply our result to the parabolic Signorini problem, one needs to know that the time derivative of the solution vanishes on the free boundary near such regular  points. To put things in perspective, it was  established  in \cite{DGPT} that the solution $u$ is $\frac{3}{4}$-H{\"o}lder continuous in the time variable via monotonicity methods. With this result alone, it would not be immediately possible to apply our Theorem \ref{T:main} to get smoothness of the free boundary. However, it was very recently established in \cite{PZ} that the time derivative  is in fact continuous near regular  points, thereby allowing the application of our result (see also the recent preprint \cite{ACM} where the same result was independently established). It remains to be seen whether one can establish real analyticity of the free boundary in the space  variable near such regular points, similar to the results obtained for the  classical parabolic obstacle problem  in \cite{KPS}.

\medskip

The paper is organized as follows. In Section \ref{S:notations} we introduce various notations and state our main results. Section \ref{S:boundaryHar} contains the proofs of these results. Finally, in Section \ref{S:HR}, we establish the higher regularity of the free boundary near regular points for the parabolic Signorini problem with zero obstacle as an application of our results.

\medskip

\textbf{Acknowledgment:} We would like to thank Prof. Nicola Garofalo and Prof. Arshak Petrosyan for  suggesting the problem and for their many helpful comments and suggestions.

\section{Notation and preliminaries}\label{S:notations}

Hereafter, when we say that a constant is universal, we mean that it depends exclusively on $n, k$ and $\alpha$.

Throughout the paper we use following notation. We write the points of $\R^n$ as $x=(x',x_n)$, where $x'=(x'',x_{n-1})\in \R^{n-1}$ and $x''\in \R^{n-2}$. Points of $\R^n\times\R$ are written as $X=(x',x_n,t)$.

For parabolic functional spaces, we use notations similar to those in \cite{Li}
and \cite{DGPT}. In particular, $H^{l,l/2}(E)$, for $l = m + \gamma, m \in \N\cup\{0\}$, $\gamma \in (0,1]$ is the space of functions such that the partial derivatives $\partial_{x}^{\alpha}\partial_t^j u$ are $\gamma$-H{\"o}lder in $x$ and $\gamma/2$-H{\"o}lder in $t$ for the derivatives of parabolic order $|\alpha|+ 2j ≤\le m$ and $(1 + \gamma)/2$-H{\"o}lder in $t$ if $|\alpha|+ 2j \le m - 1$. $L_p(E)$ stands for the Lebesgue space, and $W^{2m,m}_p(E)$ is the Sobolev space of functions
such that $\partial_{x}^{\alpha}\partial_t^j u \in L_p(E)$ for $|\alpha|+ 2j \le 2m$. We also denote by $W^{1,0}_p(E)$ the Banach subspaces of $L_p(E)$ generated by the norm

\[
||u||_{W^{1,0}_p(E)} = ||u||_{L_p(E)} +||\nabla u||_{L_p(E)}.
\]

We also define H{\"o}lder spaces as follows. Given $\beta\in (0,2]$ and $f$ defined on $\Omega\subset \R^{n+1}$, we let, for $(x_0,t_0)\in\Omega$,
\[
\langle f\rangle_{\beta;(x_0,t_0)}:=\sup\left\{\frac{|f(x_0,t)-f(x_0,t_0)|}{|t-t_0|^{\beta/2}} \ : \ (x_0,t)\in \Omega\setminus \{(x_0,t_0)\}\right\},
\]
and $\langle f\rangle_{\beta;\Omega}:=\sup_{(x_0,t_0)\in\Omega}\langle f\rangle_{\beta;(x_0,t_0)}$. For any $a>0$, we write $a=k+\alpha$, where $k$ is a nonnegative integer and $\alpha\in (0,1]$, and we define
\[
\langle f\rangle_{a;\Omega}:=\sum_{|\beta|+2j=k-1}\langle D_x^{\beta}D_t^j f\rangle_{\alpha+1},
\]
\[
[f]_{a;\Omega}:=\sum_{|\beta|+2j=k}[D_x^{\beta}D_t^j f]_{\alpha},
\]
\[
|f|_{\alpha;\Omega}:=\sum_{|\beta|+2j\le k}|D_x^{\beta}D_t^jf|+[f]_a+\langle f\rangle_a,
\]
and we let $H^a(\Omega):=\{f \ : \ |f|_a<\infty\}$.
We consider domains in $\R^n \times \R$ from which we remove an n-dimensional ``slit'' 
\[
\mathcal{P} = \{(x,t) \in \R^n\times \R \ | \  x_n =0, \ x_{n-1} \leq g(x'',t) \},
\]
where $g$ is assumed to be in $H^{k+1+\alpha}$. In particular, we will assume  $g(0)=0$, $\nabla_{x',t}g(0)=0$, and $\|g\|_{H^{k+1+\alpha}} \leq 1$ for $k\geq1$.  When $k=0$, we additionally  assume that $\|g\|_{C^{1+\alpha}} \leq 1$. We further define
\[
\Gamma = \{(x,t) \in \R^n\times\R \ | \ x_n =0, \ x_{n-1}=g(x'',t)\},
\]
and 
\[
\Psi_r = \{(x',x_n,t) : -r^2 < t \leq 0,  |x_n|<2r, |x'|<r \} \subseteq \R^n \times \R.
\]
We also let 
\[
 \Psi_{r}^{+}= \Psi_{r} \cap \{ x_n > 0\},
\]
with corkscrew point $\underline{A_r} = (0,2r,-(1+\mu)r^2)$. 

Given $X=(x',x_n,t) \in \R^n \times \R$, we denote by $d$ the signed Euclidean distance in $\R^{n-1}$ from $x'$ to $\Gamma_t:=\{y': (y', t) \in \Gamma \}$, with $d>0$ in the $e_{n-1}$ direction. We define 
\[
r = \sqrt{x_n^2 + d^2}
\]
to be the Euclidean distance in $\R^n \times \R$ from $(x',x_n,t)$ to $\Gamma_t$.

\begin{remark}\label{tbounds}
Note that from  the regularity  assumptions on $g$, it follows that $|d_t| \leq C ||g||_{C^{1+ \alpha}}$.
\end{remark}

\begin{definition}
The order of a monomial $x'^a r^b t^c$ is defined as $|a| + b+ 2c$ where for a multi-index $a=(a_1, ....a_{n-1})$,  we define  $|a|= \Sigma_{i=1}^{n-1} a_i$. The degree of a polynomial  $P$ in $(x', t, r)$ is defined to be the order of the highest order non-zero monomial in the polynomial expression for $P$.
\end{definition}

\begin{definition}\label{d:Hkalpha}
Let $k\ge 0$. We say that a function $f$ is \emph{pointwise $H^{k+\alpha}$ in the $(x',t,r)$ variables at $0 \in \Gamma$} if there exists a (tangent)  polynomial $P_0(x',t,r)$ of parabolic degree $k$ such that 
\[
f(X)=P_0(x',t,r) + O(|X|^{k+\alpha}).
\]
We denote this by $f \in H^{k+\alpha}_{x'tr}(0)$, and define $\|f\|_{H^{k+\alpha}_{x'tr}(0)}$ as the smallest $M$ for which both $\|P_0\| \leq M$ and $|f(X) - P_0(x',t,r)| \leq M|X|^{k+\alpha}$. 
Similarly, we define $f \in H^{k+\alpha}_{x'tr}(Y)$ for any $Y \in \Gamma$. 

Given $K\subset \Gamma$, we say that $f \in H^{k+\alpha}_{x'tr}(K)$ if $M := \sup_{Y \in K}{\|f\|_{H^{k+\alpha}_{x'tr}(Y)}} < \infty$ and denote this by $\|f\|_{H^{k+\alpha}_{x'tr}(K)} = M$.
\end{definition}
\begin{remark}
This notion of $H^{k+\alpha}$  coincides with the standard notion.
\end{remark}

Lastly, let $H$ be the heat operator $Hu = \Delta u - u_t$ in $\R^n\times\R.$ With the assumptions on $\mathcal{P}$ and $g$ as  above, we can now state our main result:

\begin{theorem}\label{T:main}
Let $k \geq 0$. Let $U>0$ be a solution of $HU=0$ and $u$ be a solution of $Hu = \frac{U_0}{r}f$ in $\Psi_1 \setminus \mathcal{P}$ where $U_0 = \frac{1}{\sqrt{2}}\sqrt{d+r}$, $f \in H^{k+\alpha}_{x'tr}(\Gamma \cap \Psi_1)$ and $\|f\|_{H^{k+\alpha}_{x'tr}(\Gamma \cap \Psi_1)} \leq 1$. Assume that $U, u \in C(\Psi_1)$, that both are even in $x_n$ and vanish continuously on $\mathcal{P}$. Furthermore, assume that  $\|u\|_{L^\infty(\Psi_1)} \leq 1$ and $U(\underline{A_{3/4}})=1$. Then 
\[
\Big\|\frac{u}{U}\Big\|_{H^{k+1+\alpha}_{x'tr}(\Gamma \cap \Psi_{1/2})} \leq C
\]
where $C=C(n,k,\alpha,U(\underline{A_{3/4}}))$.
\end{theorem}

\begin{remark}
In the case $k=0$, though we assume $C^{1+\alpha}$ regularity in both space and time as  opposed to $H^{1+\alpha}$ regularity, this does not prevent application of the result to the parabolic Signorini problem, as we will show.
\end{remark}

We will treat the case $k=0$ separately after dealing with $k\geq 1$. The main ingredient of the proof of Theorem \ref{T:main} in the case $k \geq 1$ is the following Schauder-type estimate:

\begin{theorem}\label{T:schauder}
Let $k \geq 1$.  Let $u\in C(\Psi_1)$ be a solution of $Hu = \frac{U_0}{r}f$ in $\Psi_1 \setminus \mathcal{P}$, where $U_0 = \frac{1}{\sqrt{2}}\sqrt{d+r}$, $f \in H^{k-1+\alpha}_{x'tr}(\Gamma \cap \Psi_1)$ and $\|f\|_{H^{k-1+\alpha}_{x'tr}(\Gamma \cap \Psi_1)} \leq 1$. Assume that $u$ vanishes continuously on $\mathcal{P}$, $\|u\|_{L^\infty(\Psi_1)} \leq 1$ and that it is even in $x_n$. Then 
\[
\Big\|\frac{u}{U_0}\Big\|_{H^{k+\alpha}_{x'tr}(\Gamma \cap \Psi_{1/2})} \leq C
\]
and 
\[
\Big\|\frac{\nabla_{x'}u}{U_0/r}\Big\|_{H^{k+\alpha}_{x'tr}(\Gamma \cap \Psi_{1/2})} \leq C
\]
where $C=C(n,k,\alpha)$.
\end{theorem}

The proof of Theorem \ref{T:schauder} relies on the following special case when $\Gamma$ is straight:

\begin{theorem}\label{T:straight}
Let $\Gamma = \{x_{n-1}=0\}$ and $u\in C(\Psi_1)$ be a solution of $Hu=0$ in $\Psi_1 \setminus \mathcal{P}$. Assume that $u$ is even in $x_n$, $\|u\|_{L^\infty(\Psi_1)} \leq 1$ and it vanishes continuously on $\mathcal{P}$. Then, for any $k \geq 0$ there exists a polynomial $P_0(x',t,r)$ of parabolic degree $k$ so that $U_0P_0$ is caloric in $\Psi_1 \setminus \mathcal{P}$ and 
\[
\Big|\frac{u}{U_0}-P_0\Big| \leq C|X|^{k+1}, \ \ \text{ with }  \ \ C=C(n,k).
\]
\end{theorem}

\section{Boundary Harnack Inequality}\label{S:boundaryHar}

In this section, we prove our main result  Theorem \ref{T:main}. We will return to the proofs of Theorems \ref{T:schauder} and \ref{T:straight} later. For now, we will use them to prove Theorem \ref{T:main}. To this end, we start by proving the following pointwise estimate:

\begin{proposition}\label{P:pointwise}
 Let $k \geq 1$. Let $0<U\in C(\Psi_1)$ be a solution of $HU=0$ in $\Psi_1 \setminus \mathcal{P}$, even in $x_n$ and vanishing continuously on $\mathcal{P}$ with $U(\underline{A_{3/4}})=1$. Let $u\in C(\Psi_1)$ be even in $x_n$ such that it vanishes on $\mathcal{P}$, $\|u\|_{L^\infty(\Psi_1)} \leq 1$, and 
 \[
 Hu(X) = \frac{U_0}{r}R(x',t,r) + F(X) \text{ in } \Psi_1 \setminus \mathcal{P},
 \]
where $|F(X)| \leq r^{-1/2}|X|^{k+\alpha}$, $R(x',t,r)$ is a polynomial of parabolic degree $k$ and $\|R\| \leq 1$. Then there is a polynomial $P(x',t,r)$ of parabolic degree $k+1$ such that 
 \[
 \Big|\frac{u}{U}-P\Big| \leq C|X|^{k+1+\alpha}
 \]
  where $C=C(n,k,\alpha,u(\underline{A_{3/4}}))$ and $\|P\| \leq C$.
\end{proposition}

\begin{proof}[Proof of Proposition \ref{P:pointwise}]

First notice that after a dilation we may assume $\|g\|_{H^{k+1+\alpha}} \leq \delta$, $|R| \leq \delta$ and $|F| \leq \delta r^{-1/2}|X|^{k+\alpha}$. 
Our first goal is to obtain an expression for $U$ in terms of $U_0$. With this in mind, we apply the Schauder estimate, Theorem \ref{T:schauder}, with $U$ in the place of $u$. Thus, near $0 \in \Gamma$, 
\[ 
U(X) = U_0(X)(P_0(x',t,r)+O(|X|^{k+\alpha})),
\]
where $P_0$ is a polynomial of parabolic degree $k$. We will show below, in Lemma \ref{L:P0}, that the constant term of $P_0$ is nonzero, hence after multiplying $U$ by a constant we can assume that 
\begin{equation}\label{UU0}
U = U_0(1+\delta Q_0+\delta O(|X|^{k+\alpha})),
\end{equation}
where $Q_0$ is of parabolic degree $k$, $\|Q_0\| \leq 1$, and the constant term of $Q_0$ is zero. 

At this point we go back to the claim that the constant term of $P_0$ is nonzero and show the nondegeneracy of $U(X)/U_0(X)$ in the following lemma:

\begin{lemma}\label{L:P0}
$P_0$ has nonzero constant term.
\end{lemma}

\begin{proof}[Proof of Lemma \ref{L:P0}]
We first show that $U \geq Cr$ near the boundary.

\medskip\noindent
\emph{Step 1.} We show that $U \geq c_0 |x_n|$ in a small slab $\Psi_{1/2} \cap \{ 0 < x_n < c_n \}$. 

First, note by the Harnack inequality that there exists $\delta_0 >0$ such that $U \geq \delta_0$ on $\Psi_{1} \cap \{ x_n = c_n \}$. Applying Lemma 11.8 in \cite{DGPT}, we obtain $U \geq c_0 |x_n|$ in $\Psi_{1/2} \cap \{ 0 < x_n < c_n \}$.

\medskip
\noindent

\emph{Step 2.} We claim that $U(x,t) \geq Cr$ in $\Psi_{1/4}^{+}$.

Notice that, since $x_n$ is not always proportional to $r$, we must adjust the point at which our estimate is centered. Let $(x,t) \in \Psi_{1/4}^{+}$ and consider the point $(x_*,t_*) = (x',x_n+d, t-d^2)$, where $d=d(x,t)$. By our estimate, $U(x_*,t_*) \geq c_0(x_n+d)$. Moreover, $x_n + d$ is proportional to $r$, hence we have $U(x_*,t_*) \geq Cr$. Moreover, from the regularity  assumptions on $\Gamma$, it follows that  the distance of  $(x_*,t_*)$  from $\Gamma$  is proportional to $r$, therefore, by interior Harnack, we then obtain that $U(x,t) \geq Cr$ for $(x,t) \in \Psi_{1/4}^{+}$.

Now we  argue by contradiction. Suppose on the contrary,  the constant term of $P_0$ is zero. Since $U_0$ grows like $r^{1/2}$, therefore for $(x, 0) \in \{ r \geq |x'|\} \cap \Psi_{1/4}^{+}$, we have that $U \leq  Kr^{3/2}$ which is inconsistent with the fact that  $U \geq Cr$ near the boundary. This establishes  the lemma.
\end{proof}

To continue the proof of Proposition \ref{P:pointwise}, we follow De Silva's and Savin's method of ``approximating polynomials'' (see \cite{DSS} and \cite{DSS14b}), which we will define in Definition \ref{d:approx}. We begin by computing some preliminary estimates. Given $m \geq 0$, notice that 
\[
\Delta r^m = mr^{m-2}(m+d\Delta d).
\]
Letting $\overline{i}$ denote the multi-index with of all zeros except a 1 in the i-th position, we have for $|\mu| + m +2s \leq k +1$ with $\mu_n = 0$ that
\begin{align*}
H(x^{\mu}r^mt^sU) & = U \Delta(x^{\mu}r^mt^s) + 2\nabla_x(x^{\mu}r^mt^s)\nabla_x U + (HU)x^{\mu}r^mt^s - Ux^{\mu}mr^{m-1}t^s\frac{d}{r}d_t -Ux^{\mu}r^mst^{s-1}\\
& = U\Big[r^mt^s\mu_i(\mu_i-1)x^{\mu - 2\overline{i}} + x^{\mu}mr^{m-2}t^s(m+d\Delta d) +2\mu_ix^{\mu-\overline{i}}mr^{m-1}t^s\frac{d}{r}d_i \\ &-x^{\mu}mr^{m-1}t^s\frac{d}{r}d_t -x^{\mu}r^mst^{s-1}\Big]
+ 2(r^m\mu_ix^{\mu-\overline{i}}t^sU_i + mx^{\mu}r^{m-1}t^sU_r)\\
& = \frac{U}{r}A + 2B,
\end{align*}
using that $U$ is caloric.

We first note that $d \in H^{k+1+ \alpha}$ in a  neighborhood of $\Gamma$ (see pages 241-243 in \cite{Li}). By Taylor expansion at the origin and recalling that $\nabla_{x'',t}g(0)=0$, we find that 
\[
d_i = \delta^i_{n-1} + \ldots,  \ \ \ \  \Delta d = \Delta d (0) + \ldots, \ \ \ \text{ and } \ \ \ d = x_{n-1} + \ldots.
\]
Thus we see that 
\begin{align*}
A &= m(m+2\mu_{n-1})x^{\mu}r^{m-1}t^s + \mu_i(\mu_i-1)x^{\mu-2\overline{i}}r^{m+1}t^s\\ &- x^{\mu}r^{m+1}st^{s-1}  + b^{\mu m}_{\sigma l}x^{\sigma}r^lt^s + \delta O(|X|^{k+\alpha}),
\end{align*}
with $b^{\mu m}_{\sigma l}$ nonzero only for $|\mu| + m -1 + 2s< |\sigma| + l +2s \leq k$. Furthermore, we have 
\[
B = \frac{U_0}{r}\Big[\frac{1}{2}r^m\mu_{n-1}x^{\mu-\overline{n-1}}t^s + \frac{1}{2}mx^\mu r^{m-1}t^s + p^{\mu m}_{\sigma l}x^{\sigma}r^lt^s + \delta O(|X|^{k+\alpha})\Big],
\] 
where $p^{\mu m}_{\sigma l}$ is nonzero only for $|\mu| + m -1 +2s< |\sigma| + l +2s \leq k$. This follows from the representations 
\[
\nabla_{x} U = \frac{U_0}{r}\Big[\frac{1}{2}P_0\nabla_x d + r(P_0)_x + (P_0)_rd\nabla_x d + O(|X|^{k+\alpha})\Big]
\]
and 
\[
U_r = \nabla_x r \nabla_x U = \frac{U_0}{r}\Big[\frac{1}{2}P_0 + (P_0)_xd\nabla_x d + r(P_0)_r + O(|X|^{k+\alpha})\Big]
\]
which are shown in Lemma \ref{L:cone} and Lemma \ref{L:DrU}. Note that the terms $b^{\mu m}_{\sigma l}x^{\sigma}r^lt^s$ and $p^{\mu m}_{\sigma l}x^{\sigma}r^lt^s$ have strictly higher degree than the first terms. Moreover, the coefficients satisfy the estimates $b^{\mu m}_{\sigma l} \leq C\delta$ and $p^{\mu m}_{\sigma l} \leq C\delta$ since they are linear combinations of the coefficients of the tangent polynomials at $0$ of $d\Delta d$, $dd_t$, and $dd_i$.

Hence we find that 
\begin{align*}
H(x^{\mu}r^mt^sU) &= \frac{U_0}{r}\Big[m(m+1+2\mu_{n-1})x^{\mu}r^{m-1}t^s + r^mt^s\mu_{n-1}x^{\mu-\overline{n-1}}\\ 
&+ \mu_i(\mu_i-1)x^{\mu-2\overline{i}}r^{m+1}t^s - x^{\mu}r^{m+1}st^{s-1} + c^{\mu m}_{\sigma l}x^{\sigma}r^lt^s + \delta O(|X|^{k+\alpha})\Big],
\end{align*}
where $c^{\mu m}_{\sigma l}$ is nonzero only for $|\mu| + m -1 +2s < |\sigma| + l +2s \leq k$ and $c^{\mu m}_{\sigma l} \leq C\delta$. Therefore, for a polynomial $P=a_{\mu m \eta}x^{\mu}r^mt^{\eta}$ of degree $k+1$, we obtain 
\[
H(UP) = \frac{U_0}{r}(A_{\sigma l s}x^{\sigma}r^lt^s + \delta O(|X|^{k+\alpha}))
\]
for $|\sigma| + l +2s \leq k$ with 

\begin{align*}
A_{\sigma l s} &= (l+1)(l+2+2\sigma_{n-1})a_{\sigma(l+1)s} + (\sigma_{n-1}+1)a_{(\sigma +\overline{n-1})ls} \\ &+(\sigma_i +1)(\sigma_i+2)a_{(\sigma+2\overline{i})(l-1)s} - (s+1)a_{\sigma(l-1)(s+1)} + c^{\mu m}_{\sigma l}a_{\mu m s}. 
\end{align*}

Notice that $a_{\sigma(l+1)s}$ can be written in terms of $A_{\sigma l s}$ and a linear combination of $a_{\mu m \eta}$ with either $|\mu| + m +2\eta< |\sigma| + l +1 +2s$, or $|\mu| + m + 2\eta= |\sigma| + l +1 +2s$ and $m < l+1$. Therefore the above equation uniquely determines the coefficients $a_{\mu m \eta}$ for given $A_{\sigma l s}$ and $a_{\mu 0 \eta}$.

We are now ready to specify what it means to be an \emph{approximating polynomial}. 

\begin{definition}\label{d:approx} We define $P$ to be \emph{approximating for $u/U$ at $0$} if the coefficients $A_{\sigma l s}$ are the same as the coefficients of $R$.
\end{definition}

The rest of the proof of Proposition \ref{P:pointwise} relies on the following Lemma:

\begin{lemma}\label{L:approx}
There exist universal constants $C$, $\rho$, and $\delta$ depending solely on $n,k,\alpha$ such that if $P$ is an approximating polynomial for $u/U$ in $\Psi_\lambda \setminus \mathcal{P}$ with $\|P\| \leq 1$ and
\[
\|u-UP\|_{L^\infty(\Psi_\lambda \setminus \mathcal{P})} \leq \lambda^{3/2+k+\alpha},
\]
then there is an approximating polynomial $\overline{P}$ for $u/U$ in $\Psi_{\rho\lambda} \setminus \mathcal{P}$ with 
\[
\|u-U\overline{P}\|_{L^\infty(\Psi_{\rho\lambda} \setminus \mathcal{P})} \leq (\rho\lambda)^{3/2+k+\alpha}
\]
and 
\[
\|\overline{P}-P\|_{L^\infty(\Psi_\lambda)} \leq C\lambda^{k+1+\alpha}.
\]
\end{lemma}

\begin{proof}[Proof of Lemma \ref{L:approx}]
First, let $u = UP + \lambda^{k+3/2+\alpha}\tilde{u}(\tilde{X})$ where $\tilde{X}=(x/\lambda,t/\lambda^2)$. Since $P$ is approximating, we have 
\[
H\tilde{u}(\tilde{X}) = \lambda^{1/2-k-\alpha}(F(X) - \delta \frac{U_0}{r} O(|X|^{k+\alpha})) = \tilde{F}(\tilde{X}).
\]
Note that by hypothesis we have $|\tilde{u}(X)| \leq 1$ and $|H\tilde{u}(X)| \leq C \delta r^{-1/2}$ in $\Psi_1$. Denote the rescalings of $\Gamma$, $\mathcal{P}$, $U_0$, and $U$ from $\Psi_\lambda$ to $\Psi_1$ by $\tilde{\Gamma}$, $\tilde{\mathcal{P}}$, $\tilde{U_0}$, and $\tilde{U}$.

We split $\tilde{u}$ into two parts, $\tilde{u}= \tilde{u}_0 + \tilde{v}$ with $H\tilde{u}_0=0$ in $\Psi_1 \setminus \tilde{\mathcal{P}}$, $\tilde{u}_0 = \tilde{u}$ on $\partial_{p} \Psi_1 \cup \tilde{\mathcal{P}}$ and $|H\tilde{v}| \leq C \delta r^{-1/2}$ in $\Psi_1 \setminus \tilde{\mathcal{P}}$, $\tilde{v} = 0$ on $\partial_{p} \Psi_1 \cup \tilde{\mathcal{P}}$. Here $\partial_{p} \Psi_1$ refers to the parabolic  boundary of $\Psi_1$. Moreover, we have the following estimate:

\begin{equation}\label{e0}
  \| \tilde{v} \|_{L^{\infty}(\Psi_1)} \leq C\delta \tilde{U}_0
\end{equation}
 In order to establish \eqref{e0}, we use as lower (upper) barriers multiples of $\overline{v} = -U_0 + U_0^2$ similar to the ones used in the proof of (5.6) in \cite{DSS14b}. In this regard , we first  note that $v_0 \leq 0$ in $\Psi_1$. Moreover from Remark \ref{tbounds} and calculations similar to those in the proof of (5.6) in \cite{DSS14b}, we have that
\begin{equation}\label{e1}
H\overline{v} \geq cr^{-1}
\end{equation}
 The above estimate \eqref{e1} implies that suitable multiples  of $\overline{v}$ can be used as  barriers to establish \eqref{e0}.  Now to estimate $\tilde{u}_0$, note that as $\delta$ tends to 0, $\tilde{\Gamma}$ converges to $\{x_{n-1}=0\}$, and $\tilde{u}_0$ is uniformly H{\"o}lder continuous in $\Psi_{1/2}$. Then by compactness, for $\delta$ small enough, we can approximate $\tilde{u}_0$ in $\Psi_{1/2}$ by a solution of the flat case $\Gamma = \{ x_{n-1}=0\}$. By Theorem \ref{T:straight} and the fact that $\tilde{U} \rightarrow \tilde{U_0}$ uniformly as $\delta \rightarrow 0$, we find $$\| \tilde{u}_0-\tilde{U}Q \|_{L^\infty(\Psi_\rho)} \leq C \rho^{k+2+1/2}$$ for a polynomial $Q$ of degree $k+1$ with $\|Q\| \leq C$. Moreover, since $U_0Q$ is caloric, we conclude from the linear system we found earlier that the coefficients of $Q$ satisfy $$(l+1)(l+2+2\sigma_{n-1})q_{\sigma (l+1) s} + (\sigma_{n-1}+1)q_{(\sigma +\overline{n-1})ls} $$ $$+(\sigma_i +1)(\sigma_i+2)q_{(\sigma+2\overline{i})(l-1)s} -(s+1)q_{\sigma (l-1)(s+1)}=0,$$ noting that the $c^{\mu m}_{\sigma l}$'s are 0 in the flat case. Now using that $ \| \tilde{v} \|_{L^{\infty}(\Psi_1)} \leq C\delta \tilde{U}_0$ we have $$\| \tilde{u}-\tilde{U}Q \|_{L^\infty(\Psi_\rho)} \leq C \rho^{k+5/2} + C\delta \leq \frac{1}{2}\rho^{k+3/2+\alpha}$$ by choosing $\rho$ and then $\delta$ sufficiently, and universally, small.

This gives us $$|u-U(P+\lambda^{k+1+\alpha}Q(\tilde{X}))| \leq \frac{1}{2}(\lambda\rho)^{3/2+\alpha+k}$$ in $\Psi_{\rho\lambda}$. But $P(X) + \lambda^{k+1+\alpha}Q(\tilde{X})$ is not quite an approximating polynomial and so we must modify our $Q$ to some $\overline{Q}$. We choose $\overline{Q}$ such that it is approximating for $R=0$, and hence its coefficients solve the system $$(l+1)(l+2+2\sigma_{n-1})\overline{q}_{\sigma (l+1)s} + (\sigma_{n-1}+1)\overline{q}_{(\sigma +\overline{n-1})ls} $$ $$ +(\sigma_i +1)(\sigma_i+2)\overline{q}_{(\sigma+2\overline{i})(l-1)s} - (s+1) \overline{q}_{\sigma (l-1)(s+1)} + \overline{c}^{\mu m}_{\sigma l}\overline{q}_{\mu m s}=0$$ with $\overline{c}^{\mu m}_{\sigma l} = \lambda^{|\sigma|+l+1-|\mu|-m}c^{\mu m}_{\sigma l}$. Note then that $|\overline{c}^{\mu m}_{\sigma l}| \leq C \delta$. Subtracting the two linear systems for the coefficients of $Q$ and $\overline{Q}$, we conclude that we can choose a $\overline{Q}$ with $$\|\overline{Q}-Q\| \leq C \delta.$$ Then taking $\delta$ small enough and setting $\overline{P} = P + \lambda^{k+1+\alpha}\overline{Q}(\tilde{X})$,  we find $$\|u-U\overline{P}\|_{L^\infty(\Psi_{\rho\lambda} \setminus \mathcal{P})} \leq (\rho\lambda)^{3/2+k+\alpha}$$ and $$\|\overline{P}-P\|_{L^\infty(\Psi_\lambda)} \leq C\lambda^{k+1+\alpha}.$$

\end{proof}

We now finish the proof of Proposition \ref{P:pointwise}. Note that since $U \geq Cr^{1/2} \geq C_1 U_0$ (which can be seen using the nondegeneracy we showed in Lemma \ref{L:P0}), the pointwise Schauder estimate  gives us that $|\tilde{u}_0| \leq C' \tilde{U}_0 \leq C U$ where $C=C(n,k,\alpha,U(\underline{A_{3/4}}))$. From \eqref{e0},  we have $\| \tilde{v} \|_{L^{\infty}(\Psi_1)} \leq C\delta \tilde{U}$ (due to the nondegeneracy). Combining these, we get $|\tilde{u}| \leq C\tilde{U}$ in $\Psi_{1/2}$, thus the hypothesis of the proposition can be improved to $$|u-UP| \leq CU\lambda^{k+1+\alpha}$$ in $\Psi_{\lambda/2}$.

After multiplying $u$ by a small constant, the hypotheses of the lemma are satisfied for some small starting $\lambda_0$. We then iterate the lemma, obtaining a limiting approximating polynomial $P_0$ with $\|P_0\| \leq 1$ and $$|u-UP_0| \leq C|X|^{k+3/2+\alpha}$$ in $\Psi_1$. By the preceding remark, we can improve the right hand side, replacing it by $CU|X|^{k+1+\alpha}$. Thus $|\frac{u}{U}-P_0| \leq C|X|^{k+1+\alpha}$.

\end{proof}

We claim that Proposition \ref{P:pointwise} implies Theorem \ref{T:main} for the case $k \geq 1$. Indeed, notice that by assumption we have $f(X) = R(x',t,r) + h(X)$, where $R$ is a polynomial of degree $k$ and $h(X) = O(|X|^{k+\alpha})$. Then $F = \frac{U_0}{r}h(X)$ satisfies the hypothesis of Proposition \ref{P:pointwise} and Theorem \ref{T:main} follows for $k \geq 1$.

Now we return to the proofs of Theorems \ref{T:schauder} and \ref{T:straight}. We start with Theorem \ref{T:straight}.

\begin{proof}[Proof of Theorem \ref{T:straight}]
We start by proving that $u$ is $C^{\infty}$ in the $x'', t$ variables. We first note that since $\mathcal{P}= \{x_{n-1}\leq 0 \}$ satisfies  the Wiener type  criterion (see (3.2) in \cite{PS}) and is scale  invariant, $u$ is  H{\"o}lder continuous in $\Psi_{1/2}$. Moreover  since the equation is invariant after differentiating in the $x''$ and $t$ variables (one can take  difference quotients in $x'',t$ as an intermediate  step and pass to the limit), we have that $\nabla_{x''} u, u_t$ are  uniformly H{\"o}lder continuous in $\Psi_{1/2}$. Now since the  norms of the derivatives are controlled by the $L^{\infty}$ norm of $u$ in the interior, repeatedly differentiating with respect to $x'', t$ (i.e. by first  taking difference  quotients) establishes that $u$ is $C^{\infty}$ in the $x'', t$ variables.

\medskip

We rewrite the equation as 
\begin{equation}\label{rewrite}
\Delta_{x_{n-1},x_n}u = -\Delta_{x''}u - u_t = f(X)
\end{equation}
and solve \eqref{rewrite} in the two dimensional planes $(x',t) \equiv$ constant. Due to invariance of the equation in $x'', t$, $u$ and $f$ have the same regularity properties.

To this end, consider the transformation $\overline{u}(z) = u(z^2)$, $\overline{f}(z) = f(z^2)$ where $z = y_{n-1} + iy_n$. Then $\overline{u}$ solves $$\Delta \overline{u} = 4|z|^2\overline{f}$$ and vanishes on $y_{n-1}=0$. After an odd reflection in $y_{n-1}$, we find that \eqref{rewrite} is satisfied with $\overline{u}, \overline{f}$ even in $y_n$ and odd in $y_{n-1}$ such that $\overline{u}$ and $\overline{f}$ have the same regularity properties.  This implies that $\overline{u}$ is $C^\infty$ in $z$. Additionally, we can expand $\overline{u}$ at $0$ as 
\[
\overline{u} = y_{n-1}(P(y_{n-1}^2,y_{n}^2)+O(|z|^{2k+2})),
\]
 for some polynomial $P$ of degree $k$. Rewriting $P$ as a polynomial in $x_{n-1} = Re z^2 = y^2_{n-1} - y^2_n$ and $r = |z|^2 = y^2_{n-1}+ y^2_n$ of degree $k$ and noticing that $U_0 = y_{n-1}$, we obtain the expansion 
 \[
 u = U_0P + U_0O(|X|^{k+1}),
 \]
 from which the desired result follows for a polynomial $P_0(x',t,r)$ after considering the $C^\infty$ dependence on the $x'',t$ variables.

To see that $U_0P_0$ is in fact caloric, we expand $P_0$ as a sum of homogeneous polynomials, $P_0= \sum_{j=0}^k p_0^j$, with the degree of $p_0^j$ being $j$, and argue by induction. 

For $j=0$, $p_0^j$ is caloric. Assume that $p_0^j$ is caloric for $j \leq i < k$ and consider 
\[
v = u - U_0 \sum_{j=0}^i p_0^j = U_0(p_0^{i+1}(x',t,r) + o(|X|^{i+1})).
\]
Notice that $v$ is caloric. Defining the rescalings $v_\lambda(X) := \frac{v(\lambda X)}{\lambda^{1/2+i+1}}$, we obtain a sequence of caloric functions converging to $U_0p_0^{i+1}$ as $\lambda\rightarrow 0$. Consequently, $U_0p_0^{i+1}$ is caloric as well, and thus by induction all $p_0^j$ are, hence $U_0P_0$ is caloric.

\end{proof}

We now return to Theorem \ref{T:schauder}. Analogously to our approach to Theorem \ref{T:main}, we start by proving a pointwise estimate.

\begin{proposition}\label{P:pointwise2}
Let $k \geq 1$ and $\Gamma \in H^{k+1+\alpha}$. Let $u\in C(\Psi_1)$ be even in $x_n$ and vanish continuously on $\mathcal{P}$. Assume that $\|u\|_{L^\infty(\Psi_1)} \leq 1$ and $H u(X) = \frac{U_0}{r}R(x',t,r) + F(X)$ in $\Psi_1 \setminus \mathcal{P}$, for $R$ a polynomial of parabolic degree $k-1$ with $\|R\| \leq 1$ and $|F(X)| \leq r^{-1/2}|X|^{k-1+\alpha}$. Then there is a polynomial $P_0(x',t,r)$ of parabolic degree $k$ satisfying 
\[
\Big|\frac{u}{U_0} - P_0\Big| \leq C|X|^{k+\alpha}
\]
and 
\[
|H(u-U_0P_0)| \leq Cr^{-1/2}|X|^{k-1+\alpha}
\]
in $\Psi_1 \setminus \mathcal{P}$  for $C=C(n,k,\alpha,u(\underline{A_{3/4}}))$ and $\|P_0\| \leq C$.
\end{proposition}

\begin{proof}[Proof of Proposition \ref{P:pointwise2}]
After an initial dilation, we can assume that $\|g\|_{H^{k+1+\alpha}} \leq \delta$, $|R| \leq \delta$, and $|F| \leq \delta r^{-1/2}|X|^{k-1+\alpha}$. We compute 
\begin{align*}
H(x^{\mu}r^mt^sU_0) &= U_0 \Big[ r^mt^s\mu_i(\mu_i-1)x^{\mu-2\overline{i}}+m(m+1)x^{\mu}r^{m-2}t^s-x^{\mu}t^s(\frac{1}{2}r^{m-1}-mdr^{m-2})\Delta d\\
& +2t^s(\frac{1}{2}r^{m-1}+mdr^{m-2})\nabla_xd\nabla_xx^{\mu} -x^{\mu}t^s(r^{m-1}+mdr^{m-2})d_t - x^{\mu}r^mst^{s-1} \Big].
\end{align*}
Recalling that $\nabla_{x'',t}g(0)=0$ and using the Taylor expansions $d_i = \delta^i_{n-1} + \ldots$, $\Delta d = \Delta d (0) + \ldots$, and $d = x_{n-1} + \ldots$, as in the proof of Proposition \ref{P:pointwise}, we obtain 
\begin{align*}
H(x^{\mu}r^mt^sU_0)&= \frac{U_0}{r}\Big[m(m+1+2\mu_{n-1})x^{\mu}r^{m-1}t^s + \mu_{n-1}x^{\mu-\overline{n-1}}r^mt^s \\ &+\mu_i(\mu_i-1)x^{\mu-2\overline{i}}r^{m+1}t^s-x^{\mu}r^{m+1}st^{s-1}+c^{\mu m}_{\sigma l}x^{\sigma}r^lt^s + \delta O(|X|^{k-1+\alpha})\Big]
\end{align*}
with $c^{\mu m}_{\sigma l} \neq 0$ only for $|\mu| + m -1 +2s< |\sigma| + l +2s \leq k-1$. Note additionally that $|c^{\mu m}_{\sigma l}| \leq C \delta$.

For a polynomial $P=a_{\mu m \eta}x^{\mu}r^mt^{\eta}$ of degree $k$, we have 
\[
H(U_0P) = \frac{U_0}{r}\left(A_{\sigma ls}x^{\sigma}r^lt^s+\delta O(|X|^{k-1+\alpha})\right)
\]
 with 
 \begin{align*}
 A_{\sigma ls} &=(l+1)(l+2+2\sigma_{n-1})a_{\sigma (l+1)s}+(\sigma_{n-1}+1)a_{(\sigma+\overline{n-1}) ls} \\ &+(\sigma_i+1)(\sigma_i+2)a_{(\sigma+2\overline{i})(l-1)s} -(s+1)a_{\sigma (l-1)(s+1)}+c^{\mu m}_{\sigma l}a_{\mu ms}.
 \end{align*}
 As in the proof of Proposition \ref{P:pointwise}, this systems determines the coefficients $a_{\mu m \eta}$ once $A_{\sigma l s}$ and $a_{\mu 0 \eta}$ are given, since $a_{\sigma (l+1)s}$ can be expressed in terms of $A_{\sigma ls}$ and a linear combination of $a_{\mu m \eta}$ with either $|\mu|+m +2\eta <|\sigma|+l+1+2s$, or $|\mu| +m +2\eta= |\sigma|+l+1+2s$ and $m<l+1$.

We define an approximating polynomial P for $u/U_0$ to be a polynomial $P=a_{\mu m \eta}x^{\mu}r^mt^{\eta}$ as above where the coefficients $A_{\sigma ls}$ coincide with the coefficients of $R$. The rest of the proof rests on an improvement of flatness lemma as in the proof of Proposition \ref{P:pointwise}:

\begin{lemma}\label{L:improflat}
There exist universal constants $C$, $\rho$, and $\delta$ depending on $n,k,\alpha$ such that if $P$ is an approximating polynomial for $u/U_0$ in $\Psi_\lambda \setminus \mathcal{P}$ such that $\|P\| \leq 1$ and  
\[
\|u-U_0P\|_{L^\infty(\Psi_\lambda \setminus \mathcal{P})} \leq \lambda^{1/2+k+\alpha},
\]
then there is an approximating polynomial $\overline{P}$ for $u/U_0$ in $\Psi_{\rho\lambda} \setminus \mathcal{P}$ with 
\[
\|u-U_0\overline{P}\|_{L^\infty(\Psi_{\rho\lambda} \setminus \mathcal{P})} \leq (\rho\lambda)^{1/2+k+\alpha}
\]
and 
\[
\|\overline{P}-P\|_{L^\infty(\Psi_\lambda)} \leq C\lambda^{k+\alpha}.
\]
\end{lemma}

\begin{proof}[Proof of Lemma \ref{L:improflat}]
First, let $u = U_0P + \lambda^{k+1/2+\alpha}\tilde{u}(\tilde{X})$ where $\tilde{X}=(x/\lambda,t/\lambda^2)$. Since $P$ is approximating, we have $$H\tilde{u}(\tilde{X}) = \lambda^{3/2-k-\alpha}\left(F(X) - \delta \frac{U_0}{r} O(|X|^{k+\alpha})\right) = \tilde{F}(\tilde{X}).$$ Note that by hypothesis we have $|\tilde{u}(X)| \leq 1$ and $|H\tilde{u}(X)| \leq C \delta r^{-1/2}$ in $\Psi_1$. Denote the rescalings of $\Gamma$, $\mathcal{P}$, $U_0$, and $U$ from $\Psi_\lambda$ to $\Psi_1$ by $\tilde{\Gamma}$, $\tilde{\mathcal{P}}$, $\tilde{U_0}$, and $\tilde{U}$.

Now we split $\tilde{u}$ into two parts, $\tilde{u}= \tilde{u}_0 + \tilde{v}$ with $H\tilde{u}_0=0$ in $\Psi_1 \setminus \tilde{\mathcal{P}}$, $\tilde{u}_0 = \tilde{u}$ on $\partial_{p} \Psi_1 \cup \tilde{\mathcal{P}}$ and $|H\tilde{v}| \leq C \delta r^{-1/2}$ in $\Psi_1 \setminus \tilde{\mathcal{P}}$, $\tilde{v} = 0$ on $\partial_{p} \Psi_1 \cup \tilde{\mathcal{P}}$. By constructing barriers, we find that $$ \| \tilde{v} \|_{L^{\infty}(\Psi_1)} \leq C\delta \tilde{U}_0.$$ (Use as barriers multiples of $\overline{v} = -U_0 + U_0^2$ as before). Now to estimate $\tilde{u}_0$, note that as $\delta$ tends to 0, $\tilde{\Gamma}$ converges to $\{x_{n-1}=0\}$, and $\tilde{u}_0$ is uniformly H{\"o}lder continuous in $\Psi_{1/2}$. Then by compactness, for $\delta$ small enough, we can approximate $\tilde{u}_0$ in $\Psi_{1/2}$ by a solution of the flat case $\Gamma = \{ x_{n-1}=0\}$. By Theorem \ref{T:straight}, we find $$\| \tilde{u}_0-\tilde{U}_0Q \|_{L^\infty(\Psi_\rho)} \leq C \rho^{k+1+1/2}$$ for a polynomial $Q$ of degree $k$ with $\|Q\| \leq C$. Moreover, since $U_0Q$ is caloric, we conclude from the linear system we found earlier that the coefficients of $Q$ satisfy $$(l+1)(l+2+2\sigma_{n-1})q_{\sigma (l+1)s} + (\sigma_{n-1}+1)q_{(\sigma +\overline{n-1})ls}$$ $$+(\sigma_i +1)(\sigma_i+2)q_{(\sigma+2\overline{i})(l-1)s} - (s+1)q_{\sigma (l-1)(s+1)}=0.$$ Now using that $ \| \tilde{v} \|_{L^{\infty}(\Psi_1)} \leq C\delta \tilde{U}_0$ we have $$\| \tilde{u}-\tilde{U}_0Q \|_{L^\infty(\Psi_\rho)} \leq C \rho^{k+3/2} + C\delta \leq \frac{1}{2}\rho^{k+1/2+\alpha}$$ by choosing $\rho$ and then $\delta$ sufficiently, and universally, small.

This gives us $$|u-U_0(P+\lambda^{k+\alpha}Q(\tilde{X}))| \leq \frac{1}{2}(\lambda\rho)^{1/2+\alpha+k}$$ in $\Psi_{\rho\lambda}$. But $P(X) + \lambda^{k+\alpha}Q(\tilde{X})$ is not quite an approximating polynomial and so we must modify our $Q$ to some $\overline{Q}$. We choose $\overline{Q}$ such that it is approximating for $R=0$, and hence its coefficients solve the system $$(l+1)(l+2+2\sigma_{n-1})\overline{q}_{\sigma (l+1)s} + (\sigma_{n-1}+1)\overline{q}_{(\sigma +\overline{n-1})ls}$$ $$+(\sigma_i +1)(\sigma_i+2)\overline{q}_{(\sigma+2\overline{i})(l-1)s} -(s+1) \overline{q}_{\sigma (l-1)(s+1)}+ \overline{c}^{\mu m}_{\sigma l}\overline{q}_{\mu m s}=0$$ with $\overline{c}^{\mu m}_{\sigma l} = \lambda^{|\sigma|+l+1-|\mu|-m}c^{\mu m}_{\sigma l}$. Note then that $|\overline{c}^{\mu m}_{\sigma l}| \leq C \delta$. Subtracting the two linear systems for the coefficients of $Q$ and $\overline{Q}$, we conclude that we can choose a $\overline{Q}$ with $$\|\overline{Q}-Q\| \leq C \delta.$$ Then taking $\delta$ small enough and setting $\overline{P} = P + \lambda^{k+\alpha}\overline{Q}(\tilde{X})$,  we find $$\|u-U_0\overline{P}\|_{L^\infty(\Psi_{\rho\lambda} \setminus \mathcal{P})} \leq (\rho\lambda)^{1/2+k+\alpha}$$ and $$\|\overline{P}-P\|_{L^\infty(\Psi_\lambda)} \leq C\lambda^{k+\alpha}.$$

\end{proof}

After multiplying $u$ by a small constant, the hypotheses of the lemma are satisfied for some small starting $\lambda_0$. We then iterate the lemma, obtaining a limiting approximating polynomial $P_0$ with $\|P_0\| \leq 1$ and 
\[
|u-U_0P_0| \leq C|X|^{k+1/2+\alpha} \ \ \text{ in } \ \ \Psi_1.
\]

The boundary  Harnack inequality gives us that $|\tilde{u_0}| \leq C \tilde{U}_0$( see \cite{PS} and the  Remark below). Since $ \| \tilde{v} \|_{L^{\infty}(\Psi_1)} \leq C\delta \tilde{U}_0$, we get $|\tilde{u}| \leq C\tilde{U}_0$ in $\Psi_{1/2}$. Thus the hypothesis of the proposition can be improved to 
\[
|u-U_0P| \leq CU_0\lambda^{k+\alpha} \ \ \text{ in } \ \ \Psi_{\lambda/2}.
\]
Consequently, we can improve the right hand side of our previous estimate by replacing it with $CU_0|X|^{k+\alpha}$, and thus \[
\Big|\frac{u}{U_0}-P_0\Big| \leq C|X|^{k+\alpha}.
\]
 Moreover, since $P_0$ is approximating, we find 
 \[
 |H(u-U_0P_0)| \leq Cr^{-1/2}|X|^{k-1+\alpha} \ \ \text{ in } \ \ \Psi_1 \setminus \mathcal{P}.
 \]
\end{proof}
\begin{remark}\label{bdh}
In the argument above,  although $U_0$ is not caloric, one can still apply boundary Harnack with $U_0$ because it is comparable to a caloric  function $H_0$ such that $H_0$ vanishes on $\mathcal{P}$. Here are the relevant details: Assume $\delta <<1$. Let $V_1= (1+ C\delta r) U_0$. Then from calculations similar to proposition 3.2 in \cite{DSS11}, we have that $V_1$ is a subsolution to the heat equation which vanishes on $\mathcal{P'}$ for $C$ sufficiently large. Similarly, $V_2= (1- C\delta r )U_0$ is a supersolution to the heat equation which also vanishes on $\mathcal{P}$. Furthermore, we can assume that $C\delta < \frac{1}{2}$ which implies 
\begin{equation}\label{cmp}
 \frac{1}{2} V_1 \leq V_2 \leq V_1
\end{equation}
 Therefore, by the Perron Process, there exists a caloric function $H_0$ which vanishes on $\mathcal{P}$ and is comparable to $U_0$  by \eqref{cmp}.
\end{remark}

\begin{proof}[Proof of Theorem \ref{T:schauder}]
First, by noting the expansion of $f$ as $f(X) = R(x',t,r) + O(|X|^{k-1+\alpha})$ where $R$ has parabolic degree $k-1$, we see that the assumptions of the proposition are satisfied. By applying the proposition, we directly obtain the first estimate in Theorem \ref{T:schauder}, namely that 
\[
\Big\|\frac{u}{U_0}\Big\|_{H^{k+\alpha}_{x'tr}(\Gamma \cap \Psi_{1/2})} \leq C.
\]

The second estimate in Theorem \ref{T:schauder} will follow by using the following estimates for the derivatives of $u$ close to $\Gamma$.

\begin{lemma}\label{L:cone}
Let $u$ be as in Proposition 2. Let $1 \leq i \leq n-1$. Then 
\[
\Big|u_i - \frac{U_0}{r}P_0^i\Big| \leq C|X|^{k-1/2+\alpha}
\]
in the cone $\{ |(x_{n-1},x_n)| \geq \text{max}(|x''|, |t|^{1/2})\}$, where $P_0^i$ has parabolic degree $k$ and $(U_0/r)P_0^i$ is obtained through formal differentiation of $U_0P_0$ at $0$ in the $x_i$ direction.
\end{lemma}

\begin{proof}[Proof of Lemma \ref{L:cone}]
We first note that in the cone $\{ |(x_{n-1},x_n)| \geq \text{max}( |x''|, |t|^{1/2})\}$, $r \geq C \text{max} (|x''|, |t|^{1/2})$. As in the proof of the previous lemma, take $\tilde{u}$ so that
\[
u-U_0P_0 = \lambda^{1/2+k+\alpha}\tilde{u}(\tilde{X}).
\]
Then $H\tilde{u}=\tilde{F}$ and $\| \tilde{u} \|_{L^\infty(\Psi_1)} \leq C$, where 
\[
\tilde{F}(\tilde{X}) = \lambda^{3/2-k-\alpha}F(X)-\frac{U_0}{r} \lambda^{3/2-k-\alpha}O(|X|^{k+\alpha}).
\]
Define $\mathcal{C} := \{|(x_{n-1},x_n)| \geq 2|x''|, 2|t|^{1/2}\} \cap (\Psi_1 \setminus \Psi_{1/4}).$ Then $\|\tilde{F}\|_{L^{\infty}(\mathcal{C})}\leq C$ and so 
\begin{equation}\label{e100}
| \nabla_{x'} \tilde{u}| \leq C \ \ \text{ in } \ \mathcal{C}' := \{|(x_{n-1},x_n)| \geq \text{max}(|x''|, |t|^{1/2})\} \cap (\Psi_{3/4} \setminus \Psi_{1/2}).
\end{equation}

Therefore, for $\lambda >0,$ $| \nabla_{x'} (u-U_0P_0) | \leq C \lambda^{k-1/2+ \alpha}$ in $\mathcal{C}'$. We also have 
\[
\nabla_{x'}(U_0P_0) = \frac{U_0}{r}\Big[\frac{1}{2}P_0\nabla_{x'}d +r\nabla_{x'}P_0 + (\partial_rP_0)d\nabla_{x'}d\Big].
\]
Since $d, \nabla_{x'}d \in H^{k+\alpha}$ we obtain
\[
\Big|\partial_i(U_0P_0) - \frac{U_0}{r}\Big[P_0^i(x',t,r)\Big]\Big| \leq C \frac{U_0}{r}|X|^{k+\alpha} \ \ \text{ in } \ \{|(x_{n-1},x_n)| \geq \text{max}(|x''|, |t|^{1/2})\},
\]
where $P_0^i$ has degree $k$. We conclude that 
\[
\Big|u_i - \frac{U_0}{r}P_0^i\Big| \leq C|X|^{k-1/2+\alpha} \ \  \text{ in } \  \{|(x_{n-1},x_n)| \geq \text{max}( |x''|, |t|^{1/2})\}.
\]

\end{proof}

To finish the proof of Theorem \ref{T:schauder}, we first note that $\tilde{F}$ is uniformly H{\"o}lder continuous in $\mathcal{P} \cap \mathcal{C}$. Since $\nabla_{x'}u$ vanishes on $\mathcal{P}$, from $C^{2, \alpha}$ Schauder estimates for $\tilde u$ in $\mathcal{C'}$, we have that $|\nabla_{x'}u | \leq Cr$. Since $U_0$ is comparable to $r^{1/2}$ in $\mathcal{C'}$, \eqref{e100} can be improved to $| \nabla_{x'} \tilde{u}| \leq C\tilde{U_0}$, and hence we can  improve our conclusion to
\begin{equation}\label{e101}
\Big|u_i - \frac{U_0}{r}P_0^i\Big| \leq C\frac{U_0}{r}|X|^{k+\alpha}
\end{equation}
in $\{|(x_{n-1},x_n)| \geq \text{max}( |x''|, |t|^{1/2})\}$, i.e,  we have estimates in non-tangential cones to $\Gamma$. The second estimate in Theorem \ref{T:schauder} follows from \eqref{e101}, by decomposing $f= R(x',t,r) + F,$ where  $R$ is a polynomial of degree at most $k$ and $F= O(|X|^{k-1+ \alpha})$, and by employing  arguments similar to Remark 5.6 in \cite{DSS14b}.

\end{proof}

\begin{lemma}\label{L:DrU}
Take $\Gamma$ and $U$ as in Proposition \ref{P:pointwise}. Then 
\[
\Big|\partial_rU - \frac{U_0}{r}P_0^r\Big| \leq C \frac{U_0}{r}|X|^{k+\alpha},
\]
where $P_0^r$ has degree $k$ and 
\[
\partial_rU = \frac{U_0}{r}\Big[\frac{1}{2}P_0 + \nabla_xP_0d\nabla d +r(D_rP_0) + O(|X|^{k+\alpha})\Big].
\]
\end{lemma}

\begin{proof}
By Lemma \ref{L:cone} above, we have that 
\[
U_i = \partial_{x_i}(U_0P_0) + O(\frac{U_0}{r} |X|^{k+\alpha})
\]
for $1\leq i \leq n-1$ and 
\[
U_n = \partial_{x_{n}}(U_0P_0) + O\Big(|X|^{k-1/2+\alpha}\Big)
\]
in the cone $\mathcal{C}_0 = \{|(x_{n-1},x_n)|>\text{max}(|x''|, |t|^{1/2})\}$. Then since $|\partial_{x_n}r| \leq r^{-1/2}U_0$, we get 
\[
\partial_rU = \partial_r(U_0P_0) + \frac{U_0}{r}O(|X|^{k+\alpha})
\]
in $\mathcal{C}_0$, and the conclusion follows by arguing as in the proof of Lemma \ref{L:cone}.
\end{proof}

At this point we prove Theorem \ref{T:main} in the case $k=0$. The proof follows the same basic strategy as the proof for the case $k \geq 1$, but we need to use regularized versions $\overline{r}$, $\overline{d}$, and $\overline{U}_0$ of $r$, $d$, and $U_0$ due to their lack of regularity. We begin with the Schauder estimate in the case $k=0$.

\begin{theorem}\label{T:schauder2}
Let $u\in C(\Psi_1)$ be a solution of $Hu = \frac{U_0}{r}f$ in $\Psi_1 \setminus \mathcal{P}$ with $|f| \leq r^{\alpha-1}$. Assume that $u$ is even in $x_n$, that it vanishes continuously on $\mathcal{P}$, $\|u\|_{L^\infty(\Psi_1)} \leq 1$ and $\Gamma \in C^{1+\alpha}$. Then
\[
\Big\|\frac{u}{U_0}\Big\|_{H^{\alpha}_{x'tr}(\Gamma \cap \Psi_{1/2})} \leq C,
\]
where $C=C(n,k,\alpha, ||u||_{L^{\infty}})$.
\end{theorem}

To prove the theorem, we will need a number of properties of the regularized functions $\overline{r}$ and $\overline{U}_0$. We state these in the following lemma, whose proof we will delay to the end.

\begin{lemma}\label{L:properties}
Let $\| \Gamma \|_{C^{1+\alpha}} \leq \delta$. Then there exist smooth functions $\overline{U}_0$ and $\overline{r}$ such that, for a universal constant $C$,
\begin{align*}
|\overline{r}-r|  &\leq C\delta r^{\alpha + 1},  &&|\overline{U}_0 - U_0| \leq C \delta U_0 r^{\alpha},\\
|\nabla \overline{r} - \nabla r | &\leq C \delta r^{\alpha},  && |\partial_{x_n} \overline{r} - \partial_{x_n} r |  \leq C \delta U_0 r^{-1/2+\alpha},\\
|H\overline{r} - \frac{1}{r} | &\leq C \delta r^{-1+\alpha},  && |H\overline{U}_0| \leq C \delta r^{-3/2+\alpha}.
\end{align*}
\end{lemma}

Now we prove the improvement of flatness lemma analogous to that in the case $k \geq 1$.

\begin{lemma}\label{L:impflat2}
Suppose $| Hu | \leq \delta r^{-3/2+\alpha}$ in $\Psi_1 \setminus \mathcal{P}$, where $u\in C(\Psi_1)$, is even in $x_n$, and vanishes continuously on $\mathcal{P}$ with $\| \Gamma \|_{C^{1+\alpha}} \leq \delta$. If there exists a constant $a$ with $|a| \leq 1$ such that 
\[
\|u-aU_0\|_{L^\infty(\Psi_\lambda)}\leq \lambda^{1/2+\alpha},
\]
 for some $\lambda >0$, then there exists a constant $b$ and $\rho>0$ such that $|a-b| \leq C \lambda^\alpha$ and
 \[
 \| u - bU_0 \|_{L^\infty(\Psi_{\rho\lambda})}\leq (\rho\lambda)^{1/2+\alpha},
 \]
 for sufficiently small $\delta$. 
\end{lemma}

\begin{proof}

By Lemma \ref{L:properties} we can assume $\|u-a\overline{U}_0\|_{L^\infty(\Psi_\lambda)}\leq 2 \lambda^{1/2+\alpha}$. Let $u = a\overline{U}_0 + 2\lambda^{1/2+\alpha}\tilde{u}(\tilde{X})$ where $\tilde{X}=(x/\lambda, t / \lambda^2)$. Then again by Lemma \ref{L:properties}, using the bound for $H\overline{U}_0$, we find that in $\Psi_1$ $ |\tilde{u}| \leq 1$ and $|H\tilde{u}| \leq C \delta r^{-3/2+\alpha}$. We split $\tilde{u}$ as $\tilde{u}=\tilde{u}_1 + \tilde{u}_2$ where 
\[
H\tilde{u}_1 = H\tilde{u} \text{  in  } \Psi_1 \setminus \mathcal{P}, \, \tilde{u}_1 = 0 \text{  on  } \partial_{p}\Psi_1 \cup \tilde{\mathcal{P}}
\] and  
\[
H\tilde{u}_2 = 0 \text{  in  } \Psi_1 \setminus \mathcal{P}, \, \tilde{u}_2 = \tilde{u} \text{  on  } \partial_{p} \Psi_1 \cup \tilde{\mathcal{P}}.
\]

Now $\| \tilde{u}_1 \|_{L^\infty} \leq C \delta \tilde{U}_0$ with $C$ not depending on $\delta$, which can be shown by using a multiple of $V=\tilde{U}_0-\tilde{U}_0^{1+2\alpha}$ as a barrier. The fact that $V$ is a barrier follows from Lemma \ref{L:properties} and calculations similar to Lemmma 5.2 in \cite{DSS14a}. Thus $\tilde{u}_1 \to 0$ uniformly as $\delta \to 0$. For $\tilde{u}_2$, by compactness  we can for $\delta$ universally and sufficiently small approximate in $\Psi_{1/2}$ by a solution of the problem in the case where $\Gamma$ is straight, and therefore by Theorem \ref{T:straight} (which we proved for $k \geq 0$), 
\[
\| \tilde{u}_2 - c\tilde{U}_0\|_{L^\infty(\Psi_\rho)} \leq C \rho^{3/2}
\]
 for a constant $c$ with $|c| \leq C$. 

As a consequence, we find that $$\| \tilde{u} - c\tilde{U}_0\|_{L^\infty(\Psi_\rho)} \leq \frac{1}{4}\rho^{1/2+\alpha}$$ and subsequently 
\[
\| u - a\overline{U}_0 - 2c\lambda^\alpha\overline{U}_0\|_{L^\infty(\Psi_{\rho\lambda})} \leq \frac{1}{2}(\rho\lambda)^{1/2+\alpha}.
\]
Applying Lemma \ref{L:properties} once more, we find 
\[
\| u  - (a+2c\lambda^\alpha) U_0\|_{L^\infty(\Psi_{\rho\lambda})} \leq (\rho\lambda)^{1/2+\alpha}.
\]

\end{proof}

Theorem 4 now follows from Lemma 7 as in the case $k \geq 1$ by iterating the Lemma above and by using boundary Harnack. In this regard, we  would like to  mention that, similarly  to the case $k\geq 1$, one can apply boundary Harnack with $U_0$. This follows from an elementary computation similar to Lemma 5.2 in \cite{DSS14a}, which shows we have that $V_1= U_0 - U_0^{1+ 2 \alpha}$ is a supersolution and $V_2= U_0+ U_0^{1+ 2\alpha}$ is a subsolution, both of which vanish on $\mathcal{P}$ and are comparable to $U_0$. Therefore, by the Perron process, there exists  a caloric function $H$ which vanishes on $\mathcal{P}$ and is comparable to $U_0$. 

We now state the $k=0$ version of Proposition \ref{P:pointwise}, from which the $k=0$ case of Theorem \ref{T:main} follows, exactly as in the $k \geq 1$ case.

\begin{proposition}\label{P:k0}
Let $0<U\in C(\Psi_1)$ be a solution of $HU=0$ in $\Psi_1 \setminus \mathcal{P}$, such that $U$ is even in $x_n$, and it vanishes continuously on $\mathcal{P}$ with $U(\underline{A_{3/4}})=1$. Let $u\in C(\Psi_1)$ be even in $x_n$, $0$ on $\mathcal{P}$ such that $\|u\|_{L^\infty(\Psi_1)} \leq 1$, and 
\[
Hu(X) = a\frac{U_0}{r} + F(X) \text{ in } \Psi_1 \setminus \mathcal{P}
\]
with $|F(X)| \leq r^{-1/2}|X|^{\alpha}$ and $|a| \leq 1$. Then there is a polynomial $P(x',r)$ of parabolic degree $1$ such that $\|P\| \leq C$ and 
\[
\Big|\frac{u}{U}-P\Big| \leq C|X|^{1+\alpha},
\]
where $C=C(n,k,\alpha)$.
\end{proposition}

Take 
\[
P(x',\overline{r})= a_0 + \sum_{i=1}^{n-1}a_ix_i + a_n\overline{r}.
\]
We compute that 
\begin{equation}\label{for}
H(UP) = a_nUH\overline{r} + 2\sum_{i=1}^{n-1} a_iU_i + 2a_n\nabla\overline{r}\nabla U.
\end{equation}

By Theorem \ref{T:schauder2}, $|u-aU_0| \leq C|X|^\alpha U_0$ for some constant $a$ with $|a| \leq C$. 

\begin{lemma}\label{L:grad}
Assume additionally that $u$ is caloric, then for a.e. $X \in \Psi_{1/2}$, 
\[
|\nabla u - \nabla (aU_0)| \leq C|X|^\alpha r^{-1/2} \text{  and  } |\nabla_{x'} u - \nabla_{x'} (aU_0)| \leq C|X|^\alpha\frac{U_0}{r}.
\]
\end{lemma}

We delay the proof of Lemma \ref{L:grad} to the end of the paper. However, with this result in hand, we can suppose after multiplying by a constant and dilating that 
\begin{equation}\label{e110}
U=U_0(1+O(\delta |X|^\alpha))
\end{equation}
and 
\begin{equation}\label{e111}
\nabla_{x'} U = \nabla_{x'}U_0 + O(\delta \frac{U_0}{r}|X|^\alpha), \; \partial_{x_n}U = \partial_{x_n} U_0 + O(\delta r^{-1/2}|X|^\alpha).
\end{equation}
Now given Lemma \ref{L:grad} (applied to $U$),  \eqref{for} , \eqref{e110}, \eqref{e111} and the estimates in Lemma \ref{L:properties}, by calculations identical to ones following (4.4) in \cite{DSS14a},  we  obtain that 
\[
H(UP) = \frac{U_0}{r}[a_{n-1} + 2a_n + O(\delta|X|^\alpha)].
\]
We define $P(x',\overline{r})$ to be an approximating polynomial for $u/U$ at the origin if $a_{n-1}+2a_n = a$. With this definition and the following lemma, whose proof is identical to that of Lemma \ref{L:improflat} for the case $k \geq 1$, Proposition 3 follows.

\begin{lemma}\label{L:existappr}
There exist universal constants $C, \rho>0$ such that if $P$ is an approximating polynomial for $u/U$ in $\Psi_\lambda \setminus \mathcal{P}$ with $\|P\| \leq 1$ and
\[
\|u-UP\|_{L^\infty(\Psi_\lambda \setminus \mathcal{P})} \leq \lambda^{3/2+\alpha},
\]
then there exists an approximating polynomial $\overline{P}$ for $u/U$ in $\Psi_{\rho\lambda} \setminus \mathcal{P}$ with 
\[
\|u-U\overline{P}\|_{L^\infty(\Psi_{\rho\lambda} \setminus \mathcal{P})} \leq (\rho\lambda)^{3/2+\alpha} \ \ \ \text{ and } \ \ \ \ 
\|\overline{P}-P\|_{L^\infty(\Psi_\lambda)} \leq C\lambda^{1+\alpha}.
\]
\end{lemma}

\begin{proof} Follows as the proof of Lemma \ref{L:improflat}.

\end{proof}

We now focus on the proofs of Lemmas \ref{L:properties} and \ref{L:grad}, which will conclude the proof of Theorem \ref{T:main} for $k \geq 0$. 

\begin{proof}[Proof of Lemma \ref{L:properties}] Recall that by assumption $||\Gamma||_{C^{1+\alpha}}\le 1$.
Notice that $d, r$ and $U_0$ are locally Lipschitz continuous, therefore are differentiable a.e. Whenever we write their derivatives, we assume we are at a point where they are differentiable.

\emph{Step 1:} We start by smoothing out the signed distance function $d$.  Define, for small $\lambda$, the following neighborhood of $\Gamma$:
\[
D_\lambda := \{X \in \mathbb{R}^{n+1} : |d(X)| < \lambda \}.
\]
Let $\rho\in C^{\infty}_0(\Psi_{1/8})$ be symmetric in $x_{n-1}$, such that $\int_{\R^{n+1}}\rho dX=1$, and define
\[ 
\rho_\lambda(X) := \lambda^{-n-1}\rho(X/\lambda), \ \ \ \ d_\lambda := d  \ast \rho_\lambda.
\]
 Since $\| \Gamma \|_{H^{1+\alpha}} \leq 1$, for a point $x_0$ on the $x_{n-1}$ axis we have 
\begin{equation}\label{dxn-1}
|d-x_{n-1}| \leq C \lambda^{1+\alpha} \ \ \text{ in } \  \ \Psi_{4\lambda},
\end{equation}
 hence $d=x_{n-1} + \lambda^{1+\alpha}v$, with $|v| \leq C$. Since $x_{n-1} \ast \rho_\lambda = x_{n-1}$, we have 
\[
d_\lambda = x_{n-1} + \lambda^{1+\alpha}(v \ast \rho_\lambda),
\]
from which we conclude that
\begin{equation}\label{dlambda0}
\nabla d_\lambda = e_{n-1} + \lambda^{1+\alpha}(v \ast \nabla \rho_\lambda), \ \ (d_{\lambda})_t = \lambda^{1+\alpha}(v\ast (\rho_{\lambda})_t), \ \ D^2 d_\lambda = \lambda^{1+\alpha} (v \ast D^2\rho_\lambda).
\end{equation}
Moreover, since 
\[
\int \lambda | \nabla \rho_\lambda | \, dX \leq C, \ \ \ \int \lambda^2 | D^2 \rho_\lambda | \, dX \leq C, \ \ \ | \nabla d(x_0) - e_{n-1} | \leq C \lambda^\alpha,
\]
we find that 
\begin{equation}\label{dlambda}
|d_\lambda - d| \leq C\lambda^{1+\alpha}, \  \  | \nabla d_\lambda - \nabla d | \leq C \lambda^\alpha, \ \ |(d_{\lambda})_t|\le C\lambda^{\alpha}, \ \ | D^2 d_\lambda | \leq C \lambda^{\alpha-1} \ \text{ in } D_{4\lambda}.
\end{equation}

We now interpolate between the $d_{\lambda}$'s with $\lambda=\lambda_l=4^{-l}$ in the annular sets $\mathcal{A}_\lambda := \{X\in \R^{n+1} \ : \  \lambda < d(X) < 4\lambda \}$. More precisely, define 
\[
\overline{d} := \psi d_\lambda + (1-\psi) d_{4\lambda},
\]
where $\psi\in C^{\infty}_0(\Psi_1)$ is such that 
\begin{equation}\label{psi}
\begin{aligned}
&\psi  = 0 \text{  for  } d>3\lambda, && \psi = 1 \text{  for  } d<2\lambda, \\
&| \nabla \psi | \leq C \lambda^{-1}, && | D^2 \psi | \leq C \lambda^{-2} \\
&|\D_{x_n}\psi|\le C\frac{|x_n|}{\lambda^2}.
\end{aligned}
\end{equation}
To obtain such a function one might take, for instance, $\psi=h\left(\frac{d_{\lambda}}{\lambda}\right)$, where 
\[
h(t)= 
\begin{cases}
1,\  t\le 9/4\\
0, \ t\ge 11/4,
\end{cases}
\]
and $h$ is smooth in between.

Notice that $\overline{d} = d_\lambda$ in $\mathcal{A}_\lambda \cap D_{2\lambda}$ and $\overline{d} = d_{4\lambda}$ in $\mathcal{A}_\lambda \setminus D_{3\lambda}$.

A direct computation using \eqref{dlambda} and \eqref{psi} leads to
\[
|\overline{d} - d | \leq C \lambda^{1+\alpha}, \ \ | \nabla \overline{d} - \nabla d | \leq C\lambda^\alpha, \ \ |D^2 \overline{d} | \leq C \lambda^{\alpha-1} \ \ \text{ in }\mathcal{A}_\lambda. 
\]
\emph{Step 2:} We smooth out $r$ in an analogous way. Define $r_\lambda := \sqrt{d^2_\lambda + x^2_{n}}$ in $\mathcal{R}_\lambda := \{X\in \R^{n+1} \ : \  \lambda/2 < r(X) < 4\lambda \}$. Note that $r$, $r_\lambda$, and $\lambda$ are all comparable in $\mathcal{R}_\lambda$ and $\mathcal{R}_{\lambda}\subseteq D_{4\lambda}$.
We have, using \eqref{dlambda},
\[
|r^2_\lambda - r^2 | = | d^2_\lambda - d^2| \leq C \lambda^{2+\alpha}.
\]
From the above equation it follows that 
\begin{equation}\label{rlambda}
|r_\lambda -r|\le C\lambda^{1+\alpha} \ \ \text{ and} \ \ \ \Big|\frac{r_\lambda}{r}-1\Big| \leq C \lambda^\alpha \ \text{ in } \mathcal{R}_{\lambda}.
\end{equation}
Since 
\begin{equation}\label{drlambda}
\nabla_x r_\lambda = \frac{1}{r_\lambda}\left(d_\lambda \nabla_{x'} d_\lambda, x_n\right),
\end{equation}
we have, using \eqref{dlambda} and \eqref{rlambda}, 
\begin{equation}\label{ds}
| \nabla r_\lambda - \nabla r| \leq C \lambda^\alpha, \ \ | D^2r_\lambda | \leq \frac{C}{\lambda}.
\end{equation}
Furthermore, \eqref{drlambda} and \eqref{dlambda0} give
\[
| \nabla r_\lambda | -1 = O(\lambda^\alpha), \ \ | \nabla d_\lambda | -1 = O(\lambda^\alpha),
\]
 which together with \eqref{dlambda} and the identity 
 \[
 r_\lambda \Delta r_\lambda + | \nabla r_\lambda |^2 = \frac{1}{2}\Delta r^2_\lambda = d_\lambda \Delta d_\lambda + | \nabla d_\lambda |^2 +1
 \]
 gives
 \begin{equation}\label{deltar}
  r_\lambda \Delta r_\lambda = 1 + O(\lambda^\alpha).
  \end{equation}
Now, \eqref{rlambda} and \eqref{deltar} give us
\[
\Big|\Delta r_\lambda - \frac{1}{r} \Big| \leq C\lambda^{\alpha-1}.
\]
Finally,
\[
(r_{\lambda})_t=\frac{d_{\lambda}}{r_{\lambda}}(d_{\lambda})_t=O(\lambda^{\alpha}),
\]
Analogously to the procedure with $d$, we iteratively glue together the $r_\lambda$'s in the annular regions $\{X\in \R^{n+1} \ : \ \lambda_l < r(X) < 4\lambda_l\}$, where $\lambda_l=4^{-l}$, by defining
\[
\overline{r} := \psi r_\lambda + (1-\psi)r_{4\lambda},
\]
where $\psi$ satisfies the properties in \eqref{psi}. Thus as above we find 
\[
|\overline{r} - r | \leq C r^{1+\alpha}, \ \ | \nabla \overline{r} - \nabla r | \leq Cr^\alpha, \ \ \Big|\Delta \overline{r} - \frac{1}{r} \Big| \leq C \lambda^{\alpha-1}, \ \ \left|H\overline{r}-\frac{1}{r}\right|\le C\lambda^{\alpha-1}.
\]

Additionally, since $|\D_{x_n}\psi|\le C\frac{|x_n|}{\lambda^2}$, then
\begin{align*}
| \partial_{x_n} \overline{r} - \partial_{x_n} r |& = \left| (r_{\lambda}-r_{4\lambda})\D_{x_n}\psi+\psi(\D_{x_n}r_{\lambda}-\D_{x_n} r)+(1-\psi)(\D_{x_n}r_{4\lambda}-\D_{x_n}r)\right|\\
& \leq C \frac{|x_n|}{r}\lambda^\alpha \leq C \frac{U_0}{r^{1/2}}\lambda^\alpha.
\end{align*}

\emph{Step 3:} We construct $\overline{U_0}$. Define 
\[
 (U_0)_\lambda := \frac{\sqrt{2}}{2}(d_\lambda + r_\lambda)^{\frac{1}{2}} \ \text{ in } \mathcal{R}_\lambda.
 \]
 We claim that $(U_0)_\lambda$ satisfies the following: 
 \begin{equation}\label{U0lambda}
 \Big|\frac{(U_0)_\lambda}{U_0} - 1 \Big| \leq C\lambda^\alpha, \ \ |\nabla(U_0)_\lambda - \nabla U_0 | \leq C \lambda^{\alpha-\frac{1}{2}}, \ \ | \Delta(U_0)_\lambda | \leq C \lambda^{\alpha-\frac{3}{2}}, \ |((U_0)_{\lambda})_t|\le C\lambda^{\alpha-\frac{3}{2}}.
 \end{equation}
Then, proceeding with the interpolation exactly as in the construction of $\overline{r}$, we obtain $\overline{U}_0$ with 
\[
|\overline{U}_0 - U_0| \leq C \delta U_0 r^{\alpha}
\]
and 
\[
| H \overline{U}_0| \leq C \delta r^{\alpha - \frac{3}{2}}.
\]

We will prove the claim in the regions $\mathcal{R}^1_\lambda := \mathcal{R}_\lambda \cap \{d> -r/2\}$ and  $\mathcal{R}^2_\lambda := \mathcal{R}_\lambda \cap \{d<-r/2\}$.

In $\mathcal{R}^1_\lambda$ we have that $U_0$, $(U_0)_\lambda$, and $\lambda^{1/2}$ are comparable and
\[
(U_0)_\lambda = U_0 \left(\frac{r_\lambda + d_\lambda}{r+d}\right)^{\frac{1}{2}}.
\]
From \eqref{dlambda}, \eqref{rlambda} and \eqref{ds} we obtain 
\[
 \frac{r_\lambda +d_\lambda}{r+d} = 1 + O(\lambda^\alpha), \ \ \nabla\left( \frac{r_\lambda +d_\lambda}{r+d}\right) = O(\lambda^{\alpha-1}),
\]
therefore
 \[
\left|\frac{(U_0)_\lambda}{U_0} - 1 \right| \leq C\lambda^\alpha, \ \ |\nabla(U_0)_\lambda - \nabla U_0 | \leq C \lambda^{\alpha-\frac{1}{2}}.
 \]
Moreover,
\[
| \nabla(U_0)_\lambda | = | \nabla U_0 | + O(\lambda^{\alpha-\frac{1}{2}}) = \frac{1}{2}r^{-\frac{1}{2}} + O(\lambda^{\alpha-\frac{1}{2}}),
\]
which, combined with the fact that 
\[
(U_0)_\lambda \Delta (U_0)_\lambda + |\nabla(U_0)_\lambda|^2 = \frac{1}{4}\Delta(d_\lambda + r_\lambda)
\]
gives
\[
| \Delta(U_0)_\lambda | \leq C \lambda^{\alpha-\frac{3}{2}}.
\]
Finally, since $d+r>\frac{r}{2}$,
\[
d_{\lambda}+r_{\lambda}=(d_{\lambda}-d)+(r_{\lambda}-r)+d+r\ge -C\lambda^{\alpha+1}+\frac{\lambda}{4}\ge \frac{\lambda}{8},
\]
for $\lambda$ small enough, hence
\[
((U_0)_{\lambda})_t=\frac{1}{2}\frac{\frac{d_{\lambda}}{r_{\lambda}}+1}{\sqrt{d_{\lambda}+r_{\lambda}}}(d_{\lambda})_t=O(\lambda^{\alpha-\frac{1}{2}}).
\]

In $\mathcal{R}^2_\lambda$ we have that $U_0$, $(U_0)_\lambda$, and $|x_n| \lambda^{-1/2}$ are comparable and 
\begin{equation}\label{U0lambda}
(U_0)_\lambda = \frac{|x_n|}{\sqrt{2}}(r_{\lambda}-d_{\lambda})^{-\frac{1}{2}}=U_0 \left(\frac{r_\lambda -d_\lambda}{r-d}\right)^{-\frac{1}{2}}.
\end{equation}
Thus one proves as above that
\[
\Big|\frac{(U_0)_\lambda}{U_0} - 1 \Big| \leq C\lambda^\alpha, \ \ |\nabla(U_0)_\lambda - \nabla U_0 | \leq C \lambda^{\alpha-\frac{1}{2}}
\]
Finally, since $\partial_{x_n} d_\lambda = 0$ and $\partial_{x_n} r_\lambda = x_n/r_\lambda$, \eqref{U0lambda} leads to (assuming $x_n>0$)
\begin{align*}
\sqrt{2}\Delta(U_0)_\lambda & = 2\partial_{x_n}(r_\lambda-d_\lambda)^{-\frac{1}{2}} + x_n\Delta(r_\lambda-d_\lambda)^{-\frac{1}{2}}\\
&= \frac{x_n}{2}(r_\lambda - d_\lambda)^{-\frac{3}{2}}\left(-2\frac{1}{r_\lambda} - \Delta(r_\lambda-d_\lambda) + \frac{3}{2}\frac{| \nabla (r_\lambda-d_\lambda)|^2}{r_\lambda-d_\lambda}\right).
\end{align*}
Consequently, since 
\[
| \nabla(r_\lambda-d_\lambda)|^2 = r_\lambda^{-2}\left(2r_\lambda(r_\lambda-d_\lambda\right) + O(\lambda^{2+\alpha})),
\]
 we obtain 
\[
| \Delta(U_0)_\lambda | \leq C \lambda^{\alpha-\frac{3}{2}}.
\]
Finally, in $R_{\lambda}^2$ we have, for $\lambda$ small,
\[
r_{\lambda}-d_{\lambda}=(r_{\lambda}-r)+(d-d_{\lambda})+r-d\ge -C\lambda^{\alpha+1}+\frac{3r}{2}\ge -C\lambda^{\alpha+1}+\frac{3\lambda}{4}\ge \frac{\lambda}{4},
\]
therefore
\[
((U_0)_{\lambda})_t=-\frac{|x_n|}{2\sqrt{2}}(r_{\lambda}-d_{\lambda})^{-\frac{3}{2}}\left(\frac{d_{\lambda}}{r_{\lambda}}-1\right)(d_{\lambda})_t=O(\lambda^{\alpha-\frac{3}{2}}).
\]
Collecting the results above, the proof of Lemma \ref{L:properties} is complete. 
\end{proof}

We turn to Lemma \ref{L:grad}.

\begin{proof}[Proof of Lemma \ref{L:grad}]
Without loss of generality, let  $X_0= (x_0, 0)$ be a point at distance $\lambda$ from $\Gamma$, and, furthermore, assume that the closest point to $X_0$ on $\Gamma$ at $t=0$ is the origin. Therefore from our assumption that the space normal at origin is $e_{n-1}$, we get that  $x_0$ belongs to the hyperplane $\{ x''=0\}$.  Let $$U_0^* = \frac{\sqrt{2}}{2}\sqrt{x_{n-1} + r^*}, \, r^* = \sqrt{x_{n-1}^2+x_n^2}.$$ Not only do $U_0^*$ and $r^*$ coincide with $U_0$ and $r$ at $X_0$, but moreover if $d$, $r$, and $U_0$ are differentiable at $X_0$, we have $\nabla d = e_{n-1}$, $\nabla U_0 = \nabla U_0^*$, and $\nabla r = \nabla r^*$ at $X_0$.

Using that $\|\Gamma\|_{C^{1+\alpha}}\leq \delta$, we find $$|U_0^* - U_0| \leq C \lambda^{1/2+\alpha}$$ in the cone $\mathcal{C} = \{ \text{max}(|x''|, |t|^{1/2}) < r^* \} \cap \{ \lambda/2 < r^* < 2\lambda \}$. So in $\mathcal{C}$, 
\begin{equation}\label{e300}
|u-aU_0^*| \leq C \lambda^{1/2+\alpha}
\end{equation}

 This is because $u-aU_0^*$ is caloric and vanishes on $\mathcal{Q}$ where $\mathcal{Q}$ is a smooth   slit which contains origin and is at a distance comparable to $\lambda$ from $X_0$. This follows from the fact that the slit $\mathcal{P}$  is $C^{1, \alpha}$ and sufficiently flat near the  origin and is comparable to $\{x_{n-1} \leq 0\}$. Therefore one can find such a $\mathcal{Q}$ for which the corresponding $U_0^{'}$ will differ from $U_0$ and $U_0^*$ by order of $\lambda^{1/2+ \alpha}$ in $\mathcal{C}$. This implies \eqref{e300}.  Then from the   gradient estimates  in $\mathcal{C}$,  we can obtain that $$ | \nabla u - a \nabla U_0^* | \leq C \lambda^{-1/2+\alpha}, \: | \nabla_{x'} u - a \nabla_{x'} U_0^* | \leq C U_0^* \lambda^{-1+\alpha}$$ at $X_0$.   that Thus replacing $U_0^*$ by $U_0$ in these inequalities we find that for arbitrary $X \in \Psi_{1/2}$ where differentiability holds, if $\pi(X)$ is the projection of $X$ on $\Gamma$ at a fixed time level and $a_{\pi(X)}$ is the corresponding constant then $$ | \nabla u - a_{\pi(X)} \nabla U_0^* | \leq C r^{-1/2+\alpha}, \: | \nabla_{x'} u - a_{\pi(X)} \nabla_{x'} U_0 | \leq C U_0 r^{-1+\alpha}.$$ Using that $| \nabla U_0| \leq Cr^{-1/2}$, $|\nabla_{x'} U_0 | \leq CU_0/r$, $r \leq |X|$, and $|a_{\pi(X)}-a| \leq C|\pi(X)|^\alpha \leq C|X|^\alpha$, the lemma is proved.
\end{proof}

\section{Higher regularity of the free boundary in the parabolic Signorini problem}\label{S:HR}
Let $\Omega\subseteq \R^n$ be a domain with a sufficiently regular boundary $\D \Omega$, and $\mathcal{M}$ be a relatively open subset of $\D \Omega$. Define $\mathcal{S}:=\D \Omega\setminus\mathcal{M}$. We consider the \emph{parabolic Signorini problem for the heat equation}, which consists of solving

\begin{align}
\Delta v -\D_tv & = 0 & \text{ in }  &\Omega_T:= \Omega\times [0,T], \label{S1}\\
v\ge\varphi, \ \ \D_{\nu}v \ge 0, \  \ (v-\varphi)\D_{\nu}v&  =  0 & \text{ on }&\mathcal{M}_T:=\mathcal{M}\times(0,T], \label{S2}\\
v &=  g & \text{ on } & \mathcal{S}_T:=\mathcal{S}\times (0,T], \label{S3}\\
v(\cdot,0) &= \varphi_0 & \text{ on } & \Omega_0:=\Omega\times\{0\} \label{S4},
\end{align}
where $\D_{\nu}$ is the outer normal derivative on $\D\Omega$ and $\varphi:\mathcal{M}_T\rightarrow \R$, $\varphi_0:\Omega_0\rightarrow \R$ and $g:\mathcal{S}_T\rightarrow \R$ are prescribed functions satisfying the compatibility conditions
\[
\varphi_0 \ge \varphi \text{ on } \mathcal{M}\times\{0\}, \ \ \ g\ge\varphi \text{ on } \D\mathcal{S}\times (0,T], \ \ \ g=\varphi_0 \text{ on } \mathcal{S}\times\{0\}.
\]
The function $\varphi$ is called the \emph{thin obstable}, since $v$ must stay above $\varphi$ on $\mathcal{M}_T$. 

We say that a function $v\in W^{1,0}_2(\Omega_T)$ is a solution of \eqref{S1}-\eqref{S4} if 
\[
v\in \mathcal{K}:=\{w\in W^{1,0}_2(\Omega_T) \ | \ w\ge \varphi \text{ on } \mathcal{M}_T, w=g \text{ on } \mathcal{S}_T\},
\]
$\D_tv\in L_2(\Omega_T), v(\cdot, 0)=\varphi_0$ and
\[
\int_{\Omega_T}\left( \langle \nabla v,\nabla (w-v)\rangle +\D_tv(w-v)\right)dx \ge 0, \ \ \forall w\in \mathcal{K}.
\]

The \emph{free boundary} is defined as
\[
\Gamma(v):=\D_{\mathcal{M}_T}\{(x,t) \in\mathcal{M}_T \ | \ v(x,t)>\varphi(x,t)\},
\]
where $\D_{\mathcal{M}_T}$ denotes the boundary in the relative topology of $\mathcal{M}_T$.

Regarding the existing literature, the reader can find the existence and uniqueness of $v$ in \cite{B}, \cite{DL}, \cite{AU} and \cite{AU2}. The H{\"o}lder continuity of the spatial derivatives $\D_{x_i}v$, for $i=1,\ldots, n$, on compact subsets of $\Omega_T \cup \mathcal{M}_T$ was proved by Athanasopoulos (see \cite{A}) and subsequently by Uraltseva in \cite{U}, and with more relaxed assumptions on the boundary data by Arkhipova and Uraltseva in \cite{AU}. 
An extensive treatment of this problem and the optimal regularity of the solution, $v\in H_{\text{loc}}^{3/2,3/4}(\Omega_T\cup \mathcal{M}_T)$, was recently proved by Danielli, Garofalo, Petrosyan and To (see Theorem 9.1 in \cite{DGPT} for a flat thin manifold $\mathcal{M}$ contained in $\R^{n-1}\times\{0\}$, assuming $\varphi\in H^{2,1}(\Omega_T)$. There the authors establish an ingenious truncated version of Poon's parabolic counterpart to Almgren's frequency formula (see \cite{P}). With a frequency formula in hands, the authors systematically classified the free boundary points by considering the limit of the generalized frequency function at the free boundary point in question.

To state the main result of this paper, we need to describe this classification. We consider
\[
\Gamma_*(v):=\{(x',t) \ | \ v(x',0,t)=\varphi(x',t), \D_{x_n}v(x',0,t)=0\}.
\]

and assume $(0,0)\in \Gamma_*(v)$,  $\varphi\in H^{l,l}(B_1\cap \R^{n-1})$, with $l=k+\gamma\ge 2$ and $0<\gamma \le 1$. Let $k\le l_0< l$, $\sigma\le l-l_0$. The classification of free boundary points is achieved by means of the
\emph{truncated frequency function}
\begin{equation}\label{frequency}
\Phi_{u_k}^{(l_0)}(r):=\frac{1}{2}re^{Cr^{\sigma}}\frac{d}{dr}\log\max\{H_{u_k}(r),r^{2l_0}\}+2(e^{Cr^{\sigma}}-1).
\end{equation}
Here
\[
H_u(r):=\frac{1}{r_2}\int_{\R^n_+\times (- r^2,0]}u(x,t)^2G(x,t)dxdt,  \text{ where } G(x, t) \text{ is the backward heat kernel on } \R^n\times \R \text{ and }
\]
\[
u_k(x,t)=\left[v(x,t)-\tilde{q}_k(x,t)-(\varphi(x',t)-q_k(x',t))\right]\psi(x),
\]
where $q_k$ is the parabolic Taylor polynomial of order $k$ of $\varphi$ at the origin, $\tilde{q}_k$ is a caloric extension polynomial of $q_k$ in $\R^n\times \R$ which is symmetric in $x_n$ and $\psi$ is a cutoff function, even in $x_n$, such that $0\le \psi\le 1, \psi=1$ on $B_{1/2}$ and $\text{supp}\ \psi\subset B_{3/4}$.
(see Section 4 of \cite{DGPT}). The frequency function \eqref{frequency} was introduced in
\cite{DGPT} as a truncated version of Poon's frequency function adjusted to the solutions of
\eqref{S1}--\eqref{S4}. By  Theorem 6.3 in \cite{DGPT}, (see also Chapter 10),
$\Phi_{u_k}^{(l_0)}(r)$ is monotone increasing and hence 
the limit 
\begin{equation}\label{truncatedhomo}
\kappa_v^{(l_0)}(0,0):=\Phi_{u_k}^{(l_0)}(0+)
\end{equation}
exists. For $(x_0,t_0)\in \Gamma_*(v)$, we let $v^{(x_0,t_0)}(x,t):=v(x_0+x,t_0+t),$ and we analogously define
\[
\kappa_v^{(l_0)}(x_0,t_0):=\kappa_{v^{(x_0,t_0)}}^{(l_0)}(0,0).
\]
$l_0$ can be pushed up to $l$ in \eqref{truncatedhomo} by setting
\[
\kappa_v^{(l)}(x_0,t_0):=\sup_{l_0<l}\kappa_v^{(l)}(x_0,t_0).
\]
The remarkable fact is that either $\kappa_v^{(l)}(x_0,t_0)=3/2$, or
$2\le \kappa_v^{(l)}(x_0,t_0)\le l $, as proved in Proposition 10.8 in \cite{DGPT}. This leads to the following definition.

\begin{definition} We say that $(x_0,t_0)\in\Gamma(v)$ is a \emph{regular} point iff
  $\kappa_v^{(l)}(x_0,t_0)=3/2$. We define
$$
\mathcal{R}(v)=\{x_0\in \Gamma_*(v)\mid \kappa_v^{(l)}(x_0,t_0)=3/2\},
$$
the set of all regular free boundary points, also known as the
\emph{regular set}.
\end{definition}

Concerning the regularity of the free boundary, it was proved in Theorem 11.3 of \cite{DGPT} that if $(0,0)$ is a regular free boundary point and $\varphi\in H^{l,l}(B_1'\times (-1,0])$, for some $l\ge 3$, where $B_r':=B_r \cap \R^{n-1}$, then $\Gamma(v)$ is given locally by the graph of a parabolically Lipschitz function function $g$ in some direction, say $e_n$. 

Moreover by an application of boundary Harnack inequality as in \cite{PS}, they showed that  there exists $\delta=\delta(v)>0$  and $\alpha > 0$ such that $\nabla_{x''} g\in H^{\alpha,\alpha/2}(B''_{\delta}\times (-\delta^2,0])$, where $B_r'':=B_r\cap \R^{n-2}$, such that, possibly after a rotation in $\R^{n-1}$,
\[
\Gamma(v)\cap \left(B_{\delta}'\times (-\delta^2,0]\right)= \mathcal{R}(v)\cap \left(B_{\delta}'\times (-\delta^2,0]\right)=\{(x',t)\in B_\delta'\times (-\delta^2,0] \ | \ x_{n-1}=g(x'',t)\}.
\]
Now very recently in \cite{PZ}, it has been obtained that $v_t$ is H{\"o}lder continuous at regular  free boundary  points.  Consequently  by applying boundary Harnack to $\frac{v_t}{v_{x_{n-1}}}$, one obtains that $g_t$ is H{\"o}lder continuous (see for instance Corollary 3.3 in \cite{PZ}, see also Theorem 4.10 in \cite{ACM}). This implies $\mathcal{R}(v)$ is a  $C^{1, \alpha}$ hypersurface in $x'$ and $t$, possibly for a different $\alpha$.

Our central result states that, in fact, $\mathcal{R}(v)$ is locally $C^{\infty}$ when $\phi \equiv 0$. Note that here, ``locally" means with respect to a backward in time parabolic cylinder of the form $B_{\delta}\times (-\delta^2,0]$, rather than in a full neighborhood of the free boundary point. Indeed, the free boundary may fail to even exist at future times.

\begin{theorem}
$\mathcal{R}(v)$ is locally $C^{\infty}$.
\end{theorem}

\begin{proof}
We have that
\begin{equation}\label{fb8}
v(x'',g(x'',t),0, t)=0.
\end{equation}
Therefore,  by differentiating equation \eqref{fb8} with respect to the variables $x_1,...x_{n-2}, t$, we  obtain that
\begin{equation}\label{fb}
\frac{D_i v}{D_{n-1} v}= D_i g,\  \frac{D_t v}{D_{n-1} v}= D_t g.
\end{equation}
 This implies that if we take  $u= D_i v$ and $U=D_{n-1} v$ in Theorem \ref{T:main}, we obtain from \eqref{fb} that $D''g\in H^{1+\alpha}$. Similarly, with $u=D_t v$ and $u=D_{n-1} v$, Theorem \ref{T:main} leads to the conclusion that $D_t g \in H^{1+\alpha}$ (note that this relies crucially on the fact that $D_t v$ vanishes on the free boundary). This implies that $g \in H^{2+\alpha}$, i.e.,  the free boundary is $H^{2+\alpha}$ regular.  We now proceed  inductively as follows. Suppose we know that $g$, and hence the free boundary, is in $H^{k+\alpha}$ for some $k \geq 2$. Then, by applying Theorem \ref{T:main} to $u=D_i v$ and $U=D_{n-1} v$, we obtain from \eqref{fb} that $D''g \in H^{k+\alpha}$. Similarly, with $u=D_t v$ and $U=D_{n-1} v$, we find that $D_t g \in H^{k+\alpha}$, implying that $g \in H^{k+1+\alpha}$. Therefore, we can repeatedly apply Theorem \ref{T:main} to conclude that $\mathcal{R}(v)$ is smooth.
\end{proof}

\def\cprime{$'$}


\begin{bibdiv}
\begin{biblist}

\bib{AC}{article}{
      author={Athanosopoulous, I},
      author={Caffarelli, L},
       title={Optimal regularity of lower dimensional obstacle problems},
        date={2004},
     journal={Zap. Nauchn. Sem. S. Peterburg, Otdel. Mat. Inst. Steklov.},
      volume={310},
}

\bib{ACM}{article}{
      author={Athanosopoulous, I},
      author={Caffarelli, L},
      author={Milakis, E.},
       title={Parabolic obstacle problems. quasi-convexity and regularity},
        date={2016},
      eprint={arXiv:1601.01516},
      status={preprint},
}

\bib{ACS}{article}{
      author={Athanosopoulous, I},
      author={Caffarelli, L},
      author={Salsa, S},
       title={The structure of the free boundary for lower dimensional obstacle
  problems},
        date={2008},
     journal={Amer J Math},
      volume={130},
      number={2},
       pages={485\ndash 498},
}

\bib{A}{article}{
      author={Athanasopoulous, Ioannis},
       title={Regularity of the solution of an evolution problem with
  inequalities on the boundary},
        date={1982},
        ISSN={0360-5302},
     journal={Comm. Partial Differential Equations},
      volume={7},
      number={12},
       pages={1453\ndash 1465},
         url={http://dx.doi.org/10.1080/03605308208820258},
      review={\MR{679950 (84m:35052)}},
}

\bib{AU}{article}{
      author={Arkhipova, A.~A.},
      author={Ural{\cprime}tseva, N.~N.},
       title={Regularity of the solution of a problem with a two-sided limit on
  a boundary for elliptic and parabolic equations},
        date={1988},
        ISSN={0371-9685},
     journal={Trudy Mat. Inst. Steklov.},
      volume={179},
       pages={5\ndash 22, 241},
        note={Translated in Proc. Steklov Inst. Math. {{\bf{1}}989}, no. 2,
  1--19, Boundary value problems of mathematical physics, 13 (Russian)},
      review={\MR{964910 (90h:35044)}},
}

\bib{AU2}{article}{
      author={Arkhipova, A.},
      author={Uraltseva, N.},
       title={Sharp estimates for solutions of a parabolic {S}ignorini
  problem},
        date={1996},
        ISSN={0025-584X},
     journal={Math. Nachr.},
      volume={177},
       pages={11\ndash 29},
         url={http://dx.doi.org/10.1002/mana.19961770103},
      review={\MR{1374941 (97a:35084)}},
}

\bib{BG}{article}{
      author={Banerjee, A.},
      author={Garofalo, N.},
       title={A parabolic analogue of the higher-order comparison theorem of de
  silva and savin},
        date={2015},
      eprint={arXiv:1503.06340},
      status={preprint,, To appear in Journal of Differential Equations},
}

\bib{B}{article}{
      author={Br{\'e}zis, Ha{\"{\i}}m},
       title={Probl\`emes unilat\'eraux},
        date={1972},
        ISSN={0021-7824},
     journal={J. Math. Pures Appl. (9)},
      volume={51},
       pages={1\ndash 168},
      review={\MR{0428137 (55 \#1166)}},
}

\bib{CFMS}{article}{
      author={Caffarelli, L},
      author={Fabes, E},
      author={Mortola, S},
      author={Salsa, S},
       title={Boundary behavior of nonnegative solutions of elliptic operators
  in divergence form},
        date={1981},
     journal={Indiana Univ. Math. J.},
      volume={30},
      number={4},
       pages={621\ndash 640},
}

\bib{CSS}{article}{
      author={Caffarelli, L},
      author={Salsa, S},
      author={Silvestre, L},
       title={Regularity estimates for the solution and the free boundary to
  the obstacle problem for the fractional laplacian},
        date={2008},
     journal={Inventiones Mathematicae},
      volume={171},
      number={2},
       pages={425\ndash 461},
}

\bib{DGPT}{article}{
      author={Danielli, Donatella},
      author={Garofalo, Nicola},
      author={Petrosyan, Arshak},
      author={To, Tung},
       title={Optimal regularity and the free boundary in the parabolic
  signorini problem},
        date={2013},
      eprint={arXiv:1306.5213},
      status={preprint, to appear in Memoirs of the Amer. Math. Soc},
}

\bib{DL}{book}{
      author={Duvaut, Georges},
      author={Lions, Jacques-Louis},
       title={Neravenstva v mekhanike i fizike},
   publisher={``Nauka'', Moscow},
        date={1980},
        note={Translated from the French by S. Yu. Prishchepionok and T. N.
  Rozhkovskaya},
      review={\MR{602699 (83h:73009)}},
}

\bib{DSS14a}{article}{
      author={De~Silva, Daniela},
      author={Savin, Ovidiu},
       title={Boundary harnack estimates in slit domains and applications to
  thin free boundary problems},
        date={2014},
      eprint={arXiv:1406.6039},
      status={preprint},
}

\bib{DSS14b}{article}{
      author={De~Silva, Daniela},
      author={Savin, Ovidiu},
       title={$c^{\infty}$ regularity of certain thin free boundaries},
        date={2014},
      eprint={arXiv:1402.1098},
      status={preprint},
}

\bib{DSS11}{article}{
      author={De~Silva, Daniela},
      author={Savin, Ovidiu},
       title={$c^{2, \alpha}$ regularity of flat free boundaries for the thin
  one-phase problem},
        date={2011},
      eprint={arXiv:1111.2513},
      status={preprint},
}

\bib{DSS}{article}{
      author={De~Silva, Daniela},
      author={Savin, Ovidiu},
       title={A note on higher regularity boundary harnack inequality},
        date={2014},
      eprint={arXiv:1403.2588},
      status={preprint},
}

\bib{JK}{article}{
      author={Jerison, D},
      author={Kenig, C},
       title={Boundary behavior of harmonic functions in nontangentially
  accessible domains},
        date={1982},
     journal={Adv. in Math.},
      volume={46},
      number={1},
       pages={80\ndash 147},
}

\bib{KN}{article}{
      author={Kinderlehrer, D},
      author={Nirenberg, L},
       title={Regularity in free boundary problems},
        date={1977},
     journal={Ann. Scuola Norm. Sup. Pisa Cl. Sci.},
      volume={4},
      number={2},
       pages={373\ndash 391},
}

\bib{KNS}{article}{
      author={Kinderlehrer, D},
      author={Nirenberg, L},
      author={Spruck, J},
       title={Regularity in elliptic free boundary problems},
        date={1978},
     journal={J. Analyse Math.},
      volume={34},
       pages={86\ndash 119},
}

\bib{KPS}{article}{
      author={Koch, Herbert},
      author={Petrosyan, A},
      author={Shi, W},
       title={Higher regularity of the free boundary in the elliptic signorini
  problem},
        date={2014},
      eprint={arXiv:1406.5011},
      status={preprint},
}

\bib{Li}{book}{
      author={Lieberman, Gary~M.},
       title={Second order parabolic differential equations},
   publisher={World Scientific Publishing Co., Inc., River Edge, NJ},
        date={1996},
        ISBN={981-02-2883-X},
         url={http://dx.doi.org/10.1142/3302},
      review={\MR{1465184 (98k:35003)}},
}

\bib{P}{article}{
      author={Poon, Chi-Cheung},
       title={Unique continuation for parabolic equations},
        date={1996},
        ISSN={0360-5302},
     journal={Comm. Partial Differential Equations},
      volume={21},
      number={3-4},
       pages={521\ndash 539},
         url={http://dx.doi.org/10.1080/03605309608821195},
      review={\MR{1387458 (97f:35081)}},
}

\bib{PS}{article}{
      author={Petrosyan, Arshak},
      author={Shi, Wenhui},
       title={Parabolic boundary harnack principle in domains with thin
  lipschitz complement},
        date={2014},
        ISSN={0020-9910},
     journal={Anal. PDE},
      volume={7},
      number={6},
       pages={1421\ndash 1463},
      review={\MR{3270169}},
}

\bib{PZ}{article}{
      author={Petrosyan, A},
      author={Zeller, A.},
       title={Boundedness and continuity of the time derivative in the
  parabolic signorini problem},
        date={2015},
      eprint={arXiv:1512.09173},
      status={preprint},
}

\bib{U}{article}{
      author={Ural{\cprime}tseva, N.~N.},
       title={H\"older continuity of gradients of solutions of parabolic
  equations with boundary conditions of {S}ignorini type},
        date={1985},
        ISSN={0002-3264},
     journal={Dokl. Akad. Nauk SSSR},
      volume={280},
      number={3},
       pages={563\ndash 565},
      review={\MR{775926 (87b:35025)}},
}

\end{biblist}
\end{bibdiv}
\end{document}